\documentclass[11pt,reqno]{amsart}
\usepackage{mathrsfs}
\usepackage{amsmath}
\usepackage{amssymb, amsthm}
\usepackage{color}
\usepackage{enumerate}
\usepackage{bm}
\usepackage[numbers,sort&compress]{natbib}
\usepackage[colorlinks,
            linkcolor=red,
            anchorcolor=green,
            citecolor=blue
            ]{hyperref}

\usepackage[pagewise]{lineno}

\numberwithin{equation}{section}

\makeatletter
\@namedef{subjclassname@2020}{2020 Mathematics Subject Classification}
\makeatother

\usepackage[left=.75 in, right=.75 in,top=.75 in, bottom=.75 in]{geometry}

\allowdisplaybreaks

\let\f=\frac

\let\Om=\Omega

\let\th=T
\let\pa=\partial

\def\dive{\mathop{\rm div}\nolimits}

\newcommand{\beq}{\begin{equation}}
\newcommand{\eeq}{\end{equation}}
\newcommand{\ben}{\begin{eqnarray}}
\newcommand{\een}{\end{eqnarray}}
\newcommand{\beno}{\begin{eqnarray*}}
\newcommand{\eeno}{\end{eqnarray*}}

\newtheorem{theorem}{Theorem}[section]
\newtheorem{definition}[theorem]{Definition}
\newtheorem{lemma}[theorem]{Lemma}
\newtheorem{proposition}[theorem]{Proposition}

\newtheorem{remark}[theorem]{Remark}
\newtheorem{Theorem}{Theorem}[section]

\newtheorem{Lemma}[Theorem]{Lemma}

\begin{document}
\title[viscous surface waves]{Stability of Navier-Stokes equations with a free surface}

\author{Xing Cheng}
\address{College of Science, Hohai University, Nanjing 210098, Jiangsu, China.}
\email{chengx@hhu.edu.cn}
\author{Yunrui Zheng}
\address{School of Mathematics, Shandong University, Shandong 250100, Jinan, P. R. China}
\email{yunrui\_zheng@sdu.edu.cn}



\begin{abstract}
We consider the viscous incompressible fluids in a three-dimensional horizontally periodic do-
main bounded below by a fixed smooth boundary and above by a free moving surface. The fluid dynamics
are governed by the Navier-Stokes equations with the effect of gravity and surface tension on the free sur-
face. We develop a global well-posedness theory by a nonlinear energy method in low regular Sobolev spaces.
We use several techniques, including: the horizontal energy-dissipation estimates, a new tripled bootstrap
argument inspired by Guo and Tice [Arch. Ration. Mech. Anal.(2018)]. Moreover, the solution decays
asymptotically to the equilibrium in an exponential rate.
\end{abstract}

\maketitle

\section{Introduction}
\subsection{Formulation of the problem}
The viscous surface waves problem is presented for a viscous fluid of finite depth below the air. Concretely we consider 
an incompressible viscous fluid in a moving domain
\[
\Om(t)= \left\{y\in\Sigma\times\mathbb{R} | -1<y_3<\eta \left(y_1,y_2,t \right) \right\}.
\]
The upper boundary of $\Om(t)$ is free and assumed to be a graph of the unknown function $\eta: \Sigma\times\mathbb{R}^+\to \mathbb{R}$.
Here we assume $\Sigma= \mathbb{T}_{L_1} \times \mathbb{T}_{L_2}$,
 where $\mathbb{T}_{L_1}$ and $\mathbb{T}_{L_2}$ are the tori with periodicity lengths being $L_1$ and $L_2$.

For each $t\ge0$, the fluid is described by its velocity and pressure functions $(u,p): \Om(t)\mapsto\mathbb{R}^3\times\mathbb{R}$ which satisfy the following Navier-Stokes equations with boundary conditions:
\beq\label{eq:NS}
\left\{
\begin{aligned}
  &\pa_tu+u\cdot\nabla u+\nabla p-\mu\Delta u=0\quad&\text{in}&\ \Om(t),\\
  &\dive u=0\quad&\text{in}&\ \Om(t),\\
  &S(p,u)\nu  =g\eta\nu-\kappa\mathcal{H}\nu\quad&\text{on}&\ \{y_3=\eta(y_1,y_2,t)\},\\
  &\pa_t\eta=u\cdot\mathcal{N}\quad&\text{on}&\ \{y_3=\eta(y_1,y_2,t)\},\\
  &u=0\quad&\text{on}&\ \{y_3=-1\}.
\end{aligned}
\right.
\eeq
In the system \eqref{eq:NS}, $S(p,u)$ is the stress tensor which is defined by $S(p,u)= pI-\mu\mathbb{D}u$ for $I$ the $3\times3$ identity matrix, $\left(\mathbb{D}u \right)_{ij}=\pa_iu_j+\pa_ju_i$ the symmetric gradient of $u$, and $\mu>0$ the viscosity. Clearly, for the incompressible fluid, $\dive S(p,u)=\nabla p-\mu\Delta u$. In the system \eqref{eq:NS}, $\mathcal{N}= \left(-\pa_1\eta,-\pa_2\eta,1 \right)^\top$ is the outward-pointing normal on $\Sigma(t)$, and $\nu=\mathcal{N}/|\mathcal{N}|$ is the outward-unit-normal on $\Sigma(t)$. $g>0$ is the strength of gravity, $\kappa>0$ is the coefficient of surface tension, $\mathcal{H}=\nabla_\ast\cdot\left(\f{\nabla_\ast\eta}{\sqrt{1+ \left|\nabla_\ast\eta \right|^2}}\right)$ is the twice mean curvature of the free surface, where $\nabla_\ast= \left(\pa_1, \pa_2 \right)$ is the horizontal gradient operator.
We refer to \cite{WL} for the analysis of boundary conditions in \eqref{eq:NS}. Note that in \eqref{eq:NS}, we have shifted the gravitational forcing to the boundary and eliminated the constant atmospheric pressure $p_{atm}$ by $p= \bar{p}+gy_3-p_{atm}$ for the actual pressure $\bar{p}$.

The problem \eqref{eq:NS} is studied when initial data $ \left(u_0,\eta_0 \right)$ satisfy some certain compatibility conditions. We always assume that $\eta_0+1>0$, that avoids the contact lines generated by the intersection of free surface and bottom. Without loss of generality, we may assume that $\mu=g=\kappa=1$ by scaling.

The system \eqref{eq:NS} possesses a natural physical energy. For sufficiently regular $(u,\eta)$, \eqref{eq:NS} possesses a relation between the change of physical energy and dissipation:
\beq\label{eq:energy-dissipation}
\begin{aligned}
 & \f12\int_{\Om(t)}|u(t)|^2+\f12\int_{\Sigma} \left|\eta(t)\right|^2+2 \left(\sqrt{1+ \left|\nabla_\ast\eta \right|^2}-1 \right)+\f12\int_0^t\int_{\Om(s)} \left|\mathbb{D}u(s) \right|^2\\
= &  \f12\int_{\Om(0)} \left|u_0 \right|^2+\f12\int_{\Sigma} |\eta_0|^2+2 \left(\sqrt{1+ \left|\nabla_\ast\eta_0 \right|^2}-1 \right).
\end{aligned}
\eeq
The first two integrals on the left-hand side represent the kinetic and potential energies, while the third one on the left-hand side represents the dissipation. This energy-dissipation relation \eqref{eq:energy-dissipation} is our basis of the nonlinear energy method.

\subsection{Geometric reformulation}
The Navier-Stokes equations \eqref{eq:NS} admit a steady state $u_s=0, \eta_s=0$, that enables us to switch the problem \eqref{eq:NS} into a fixed domain in order to remove the difficulties created by the free surface $\Sigma(t)$ and deformable domain $\Om(t)$. We reformulate \eqref{eq:NS} as in \cite{GT3}. Now we consider the fixed equilibrium domain
\[
\Om:= \left\{x\in\Sigma\times\mathbb{R}| -1<x_3<0 \right\}.
\]
We will think of $\Sigma$ as the upper boundary of $\Om$ without confusion, and $\Sigma_b:=\{x_3=-1\}$ as the lower boundary. Then we define the harmonic extension of $\eta$:
\beq
\bar{\eta}=
\sum_{n\in L_1^{-1}\mathbb{Z}\times L_2^{-1}\mathbb{Z}} e^{2\pi in\cdot x^\prime}e^{ 2\pi|n|x_3 }\hat{\eta}(n), \quad x^\prime\in
\mathbb{T}_{L_1} \times \mathbb{T}_{L_2}, 
\eeq
where $\hat{\eta}(n)$ denotes the Fourier transform of $\eta(x^\prime)$ in discrete form.

Now we define the mapping $\Phi:\Om\to\Om(t)$ to be
\beq\label{def:map}
\Om\ni\left(x_1,x_2,x_3 \right)\mapsto\left(x_1,x_2,x_3+\bar{\eta}\left(1+ x_3 \right)\right) =  \left(y_1,y_2,y_3 \right)\in\Om(t).
\eeq
Clearly, $\Phi\left(\Sigma,t \right)= \left\{y_3= \eta\left(y_1,y_2,t \right) \right\}$ and $\Phi \left(\Sigma_b,t \right)= \left\{y_3=-1 \right\}$; that is, $\Phi$ maps $\Sigma$ and $\Sigma_b$ to the free surface and lower boundary of $\Om(t)$ respectively.
The Jacobi matrix $\nabla\Phi$ and transform matrix $\mathcal{A}$ are
\beq
\nabla\Phi=\left(
\begin{array}{ccc}
  1&0&0\\
  0&1&0\\
  A&B&J
\end{array}\right),\quad \mathcal{A}= \left(\nabla\Phi \right)^{-\top}=\left(
\begin{array}{ccc}
  1&0&-A K\\
  0&1&-BK\\
  0&0&K
\end{array}
\right),
\eeq
where
\beq\label{components}
A=\pa_1\bar{\eta} W,\ B=\pa_2\bar{\eta} W,\ J=1+\bar{\eta}+\pa_3\bar{\eta} W,\ K=\f1{J},\ W=1+x_3.
\eeq
For any scalar-valued or vector valued function $f$ and vector $X$, we define the transformed operators as follows:
\beq
\begin{aligned}
\left(\nabla_{\mathcal{A}}f \right)_i=\mathcal{A}_{ij} \pa_jf,\
\nabla_{\mathcal{A}}\cdot X=\mathcal{A}_{ij}\pa_jX_i,\ \Delta_{\mathcal{A}}f =\nabla_{\mathcal{A}}\cdot \nabla_{\mathcal{A}}f,\\
S_{\mathcal{A}}(p,u)= pI -\mathbb{D}_{\mathcal{A}}u,\ \left(\mathbb{D}_{\mathcal{A}}u \right)_{ij}=\mathcal{A}_{ik}\pa_ku_j+\mathcal{A}_{jk}\pa_ku_i,
\end{aligned}
\eeq
where $I$ is the $3\times3$ identity matrix and the summation is understood in the Einstein convection.

If $\eta$ is sufficiently small (in an appropriate sense), the mapping $\Phi$ is a diffeomorphism. In the new coordinates, for each $t\ge0$, the original system \eqref{eq:NS} becomes
\beq\label{eq:new_ns}
\left\{
\begin{aligned}
&\pa_tu-\pa_t\bar{\eta}WK\pa_3u+  u\cdot\nabla_{\mathcal{A}}u+\nabla_{\mathcal{A}}p-\Delta_{\mathcal{A}}u=0 
\ &\text{in}&\ \Om,\\
  &\dive_{\mathcal{A}}u=0\ &\text{in}&\ \Om,\\
  &S_{\mathcal{A}}(p,u)\mathcal{N}=\eta\mathcal{N}-\mathcal{H}\mathcal{N}\ &\text{on}&\ \Sigma,\\
  &u=0\ &\text{on}&\ \Sigma_b,\\
&\pa_t\eta=u\cdot\mathcal{N}\ &\text{on}&\ \Sigma,\\
&u(x,0)=u_0(x),\ \eta \left(x^\prime,0 \right)=\eta_0 \left(x^\prime \right).
\end{aligned}
\right.
\eeq
Here we still write $\mathcal{N}= \left(-\pa_1\eta,-\pa_2\eta,1 \right)$ for the nonunit-normal to $ \left\{y_3=\eta \left(y_1,y_2,t \right) \right\}$.
If we extend $\dive_{\mathcal{A}}$ to act on symmetric tensors in the natural way, 
then $\dive_{\mathcal{A}}S_{\mathcal{A}}(p,u)=\nabla_{\mathcal{A}}p-\Delta_{\mathcal{A}}u$ for $u$ satisfying $\dive_{\mathcal{A}}u=0$.
Clearly, all the spatial differential operators are related to the geometric structure of the free surface $\eta$.

\subsection{Main results}
To state our result, we need to explain the notation for spaces and norms. When we write $ \left\|\pa_t^ju \right\|_{H^k}$, and $ \left\|\pa_t^jp \right\|_{H^k}$, we always mean that the space is $H^k(\Om)$, and when we write $ \left\|\pa_t^j\eta \right\|_{H^s}$, we always mean that the space is $H^s(\Sigma)$, where $H^k(\Om)$ and $H^s(\Sigma)$ are usual Sobolev spaces for $k, s\ge0$.

When the coefficient of viscosity $\mu$ is zero, the system \eqref{eq:NS} is reduced to the free boundary problems of Euler equations with vorticity. There are
many fruitful results for well-posedness and stability of \eqref{eq:NS} when $\mu = 0$, such as \cite{BSWZ16, SZ11, SZ08, WZZZ21, CL00} and references therein. Many of these results are related to dispersive estimates.

The viscous surface waves have been studied intensively over a long period of time, since the significant progress made by Beale \cite{Beale1}.
 We first refer to the case of absence of surface tension.
When $\Sigma= \mathbb{R}^2$ or $\mathbb{T}^2$, Beale \cite{Beale1} proved the local well-posedness in Sobolev-type spaces for any fixed bottom (not even flat), and Sylvester \cite{S90} proved the global well-posedness in the same framework of Beale.
Recently, Gui \cite{G21} proved the global well-posedness using the Lagrangian structure.
When the bottom is flat as in our setting, Hataya \cite{Ha09} proved the global well-posednss for $\mathbb{T}^2$, while
Hataya and Kawashima \cite{HK09} proved global well-posedss for $\mathbb{R}^2$.
Both \cite{Ha09} and \cite{HK09} assumed that $\left(u_0,\eta_0 \right)\in H^r(\Om)\times H^{r+1/2}(\Sigma)$ for $4< r<9/2$ were small and obtained the solutions with polynomial decay in time. The new breakthrough is due to Guo and Tice \cite{GT1,GT2,GT3}.
They developed a new two-tier energy method to prove the global well-posedness under the assumption that the initial data $ \left(u_0, \eta_0 \right)\in H^{4N}(\Om)\times H^{4N+1/2}(\Sigma)$ with $N\ge3$.
They proved the solution decays to the equilibrium at a polynomial rate when $\Sigma=\mathbb{R}^2$ and at an almost exponential rate when $\Sigma=\mathbb{T}^2$.  Recently, Ren, Xiang and Zhang \cite{RXZ19} used a moving domain iteration introduced in \cite{WZZZ21}
and obtained the local well-posedness without surface tension for large data $ \left(u_0,\eta_0 \right)\in H^2(\Om)\times H^{5/2}(\Sigma)$.

For the viscous surface waves with surface tension, Beale \cite{Beale2} proved the global well-posedness in Sobolev-type spaces using the Fourier-Laplace method with the assumption that $ \left(u_0,\eta_0 \right)\in H^r(\Om)\times H^{r+1/2}(\Sigma)$ with $5/2<r<3$ when $\Sigma=\mathbb{R}^2$.
Then Beale and Nishida \cite{BN85} showed the solution in \cite{Beale2} decays to the equilibrium in an optimal polynomial rate.
 Bae \cite{B01} proved the global well-posedness by an energy method with the same assumption as in \cite{Beale2} when $\Sigma=\mathbb{R}^2$.
Nishida, Teramoto, and Yoshihara \cite{NTY04} proved the global well-posedness with the assumption that $(u_0,\eta_0)\in H^r(\Om)\times H^{r+1/2}(\Sigma)$ with $2<r<5/2$ when $\Sigma=\mathbb{T}^2$ and proved that the solutions decay to the equilibrium in the exponential rate.
Tan and Wang \cite{TW14} considered the zero surface tension limit under the regularity assumption that the initial data $ \left(u_0, \eta_0 \right)\in H^{4N}(\Om)\times H^{4N+1/2} (\Sigma)$ with $N\ge3$.

We now consider the case $\Sigma=\mathbb{T}^2$ with surface tension. To present our results, we first define the energy
\begin{align}
\label{def:energy}
\mathcal{E}: & =\|u\|_{H^2(\Om)}^2+ \left\|\pa_tu \right\|_{H^0(\Om)}^2+\|p\|_{H^1(\Om)}^2
 + \|\eta\|_{H^3(\Sigma)}^2+ \left\|\pa_t\eta \right\|_{H^{3/2}(\Sigma)}^2,
\intertext{ and  dissipation}
\mathcal{D}: & =
\|u\|_{H^3(\Om)}^2+ \left\|\pa_tu \right\|_{H^1(\Om)}^2
+\|p\|_{H^2(\Om)}^2+\|\eta\|_{H^{7/2}(\Sigma)}^2+ \left\|\pa_t\eta \right\|_{H^{5/2}(\Sigma)}^2+ \left\|\pa_t^2\eta \right\|_{H^{1/2}(\Sigma)}^2.\label{def:dissipation}
\end{align}
Then we denote $\mathcal{A}_0=\mathcal{A}(t=0)$, $\mathcal{N}_0=\mathcal{N}(t=0)$ and $\Pi_0(v)$ to be the projection of vector $v$ on the surface $ \left\{x_3=\eta_0 \right\}$. Our main result is stated as
\begin{theorem}\label{thm:main}
Let $\eta_0+1>\delta_0>0$ for some $\delta_0>0$ and $ \left(u_0,\eta_0 \right)\in H^2(\Om)\times H^{3}(\Sigma)$ satisfy the compatibility condition
  \beq\label{cond:compatibility}
\left\{
\begin{aligned}
   &\dive_{\mathcal{A}_0}(u_0)=  0\quad&\text{in}&\ \Om,\\
   &\Pi_0\left(\mathbb{D}_{\mathcal{A}_0}u_0\mathcal{N}_0 \right)=0\quad&\text{on}&\ \Sigma, \\
   &u_0=0\quad   &\text{on}&\ \Sigma_b,
\end{aligned}
\right.
  \eeq
and the zero average condition $\left<\eta_0\right>=\f1{L_1L_2}\int_\Sigma\eta_0=0$. There exists a small $\delta>0$ such that if $\mathcal{E}(0)\le \delta$,
then \eqref{eq:new_ns} admits a unique global in time solution $(u,p,\eta)$ achieving the initial data.
Moreover, the solution satisfies 
  \[
  \sup_{t\ge 0}e^{\sigma t} \mathcal{E}(t)+\int_0^\infty \mathcal{D}(t)\,\mathrm{d}t\le C \mathcal{E}(0)
  \]
  for some universal constants $\sigma = \sigma(L_1, L_2)>0$ and $C>0$.
\end{theorem}
\begin{remark}
   From the divergence free condition, the zero average is conserved in all time:
\[
\f{d}{dt}\int_\Sigma\eta(t)=\int_\Sigma u\cdot\mathcal{N}=\int_\Om \dive_{\mathcal{A}}u J=0.
\]
The zero average condition is natural. If it fails, we need to change the domains and unknowns like
  \[
  \Om\mapsto  \Om- \left\{0,0,\left<\eta_0\right> \right\},\ \eta\mapsto\eta-\left<\eta_0\right>,\ \text{and}\ p\mapsto p-\left<\eta_0\right>.
  \]
  The new unknowns still satisfy \eqref{eq:new_ns} in the new domain under the zero average condition.
\end{remark}
\begin{remark}
  The initial data in $\mathcal{E}(0)$, except for $u_0$ and $\eta_0$, are constructed in the proceeding of local well-posedness. The smallness of $\eta$ is to guarantee that the mapping $\Phi$ defined in \eqref{def:map} is a $C^1$ diffeomorphism, which enables us to switch the solutions in \eqref{eq:new_ns} into solutions in \eqref{eq:NS} so that \eqref{eq:NS} admits a unique global in time and decay solution. The proof of diffeomorphism is also contained in the proof of the local well-posedness.
\end{remark}
\begin{remark}
  Our main result yields the asymptotic stability of solutions in \eqref{eq:new_ns} around the equilibrium with exponentially decaying rate.
\end{remark}
 The big difference between the progress in \cite{NTY04} and our main result is the regularity of initial data. The regularity of $u_0$ is weakened to $H^2(\Om)$. We assume that $\eta_0\in H^3(\Sigma)$, since it is naturally observed from the following variants of \eqref{eq:energy-dissipation}
 \[
 \begin{aligned}
& \int_{\Om(t)} |u(t)|^2+\int_{\Sigma} \left|\eta(t) \right|^2+ \f{ 2 \left|\nabla_\ast\eta \right|^2}{\sqrt{1+ \left|\nabla_\ast\eta \right|^2}+1}+\int_0^t\int_{\Om(s)} \left|\mathbb{D}u(s) \right|^2
=  \int_{\Om(0)}|u_0|^2+ \int_{\Sigma} \left|\eta_0 \right|^2+\f{  2 \left|\nabla_\ast\eta_0 \right|^2}{\sqrt{1+ \left|\nabla_\ast\eta_0 \right|^2}+1},
\end{aligned}
 \]
that $\eta$ formally gains one more derivative than $u$ under the assumption of small data. Moreover, we still could prove the global well-posedness under our assumption. In addition, the assumption of initial data is consistent with our definition of energy defined in \eqref{def:energy}. In \cite{NTY04}, the authors required much more regularity to prove the exponential decay rate for $\|u\|_2^2+\|\eta\|_3^2$ due to their parabolic perturbation method. But in our result, the decay result is naturally included in the decay of energy. The idea of our method is borrowed partially from \cite{GT3} and \cite{ZhI17}.

\subsection{Strategy of the proof}
Our proof by the nonlinear energy method is based on a higher-regular formulation of energy-dissipation equality \eqref{eq:energy-dissipation} and a bootstrap argument of elliptic equations with variable coefficients.

\textbf{The horizontal energy-dissipation estimates}. The slab domain $\Om$ determines we can only take horizontal derivatives $\pa_t$ and $\nabla_\ast$ to preserve the form of energy-dissipation equality \eqref{eq:energy-dissipation}. Actually, we will take one temporal derivative and horizontal derivatives up to second order for \eqref{eq:new_ns} to estimate our energy, which is comparable with the fact that one temporal derivative behaves like two spatial derivatives. The choice of our energy and dissipation seems optimal in our framework, and could not be improved more.

The differential operators in \eqref{eq:new_ns} could not commute with the horizontal differential operators because of the variable coefficients. So there will be many nonlinear terms, which could be controlled by the horizontal energy, dissipation and energy, dissipation defined in \eqref{def:energy} and  \eqref{def:dissipation}. Then the following elliptic estimates will be combined to close our energy.

\textbf{Elliptic estimates}.
 Let us first explain the usual methods of treating elliptic estimates, such as \cite{GT1,TW14}, etc., which is different from ours here.
 The method to deal with the elliptic estimates is usually employed by decoupling $(u,p)$ with $\eta$.
In particular, for any $\eta\in H^3$, we will apply the elliptic theory developed in \cite{ADN} to 
\beq\label{eq:elliptic_1}
\left\{
\begin{aligned}
  &-\Delta u+\nabla p=\pa_tu+\mathfrak{F}^1\quad&\text{in}&\ \Om,\\
  &\dive u=\mathfrak{F}^2\quad&\text{in}&\ \Om,\\
  &S(p,u)e_3=(\eta-\Delta_\ast\eta)e_3+\mathfrak{F}^3\quad&\text{on}&\ \Sigma,\\
  &u=0\quad&\text{on}&\ \Sigma_b
\end{aligned}
\right.
\eeq
and obtain the estimate for $\|u\|_{H^2}$ and $\|p\|_{H^1}$. 
Here $\mathfrak{F}^i$, $i=1, 2, 3$, are perturbation terms and $\Delta_\ast=\pa_1^2+\pa_2^2$. 
Then to enhance the regularity of $(u,p,\eta)$, it is concerned with the system
\beq\label{eq:elliptic_2}
\left\{
\begin{aligned}
&-\Delta u+\nabla p=\pa_tu+\mathfrak{F}^1\quad&\text{in}&\ \Om,\\
&\dive u=\mathfrak{F}^2\quad&\text{in}&\ \Om,\\
&u=u&\ \Sigma,\\
&u=0\quad&\text{on}&\ \Sigma_b.
\end{aligned}
\right.
\eeq
By applying the horizontal energy-dissipation estimates to obtain the bounds $\|u\|_{H^3}^2+\|p\|_{H^2}^2$. 
Then utilizing the equation $S(p,u)e_3=(\eta-\Delta_\ast\eta)e_3+\mathfrak{F}^3$ to bound $\|\eta\|_{7/2}$. 
Honestly, this method is effective for the higher regular Sobolev spaces, and also has no harm to our a priori estimates. 
Unfortunately, it is bad for the local well-posedness of our case, especially in the construction of approximate solutions, 
the energy estimates might not be closed due to the limitation of regularity of $\eta$. 
The main reason is probably that $\|u\|_{H^2}$ and $\|u\|_{H^3}$ can not be constructed independently, see \cite{GT1}, 
while the systems \eqref{eq:elliptic_1} and \eqref{eq:elliptic_2} work independently so that they may run into circular proof in the local theory.

One way to conquer this difficulty comes from our experience of research on contact lines \cite{ZhI17}, where we handle $(u,p)$ coupling with $\eta$.
This is natural since $(u,p,\eta)$ are mutually dependent in the original system \eqref{eq:new_ns},
which might probably not be handled separately in the much more critical spaces. 
So we construct the elliptic estimates for the triple $(u,p,\eta)$ and obtain the bounds $\|u\|_{H^2}$ and $\|u\|_{H^3}$ in the same system,
not in the different systems \eqref{eq:elliptic_1} and \eqref{eq:elliptic_2}, respectively. 
This structure plays an important role in the a priori estimates and local well-posdness for large data. 
As far as we know, this view point has not been applied to consider the viscous surface waves previously.

\subsection{Notation and terminology}

Now, we mention some definitions, notation and conventions that we will use throughout this paper.

\begin{enumerate}[1.]
  \item Constants. The constant $C>0$ will denote a universal constant that only depends on the parameters of the problem, $N$ and $\Om$, but does not depend on the data, etc. They are allowed to change from line to line. We will write $C=C(z)$ to indicate that the constant $C$ depends on $z$. And we will write $a\lesssim b$ to mean that $a\le C b$ for a universal constant $C>0$.

  \item Norms.
      Sometimes we will write $\|\cdot\|_k$ instead of $\|\cdot\|_{H^k(\Om)}$ or $\|\cdot\|_{H^k(\Sigma)}$. We assume that functions have natural spaces. For example, the functions $u$, $p$, and $\bar{\eta}$ live on $\Om$, while $\eta$ lives on $\Sigma$. So we may write $\|\cdot\|_{H^k}$ for the norms of $u$, $p$, and $\bar{\eta}$ in $\Om$, and $\|\cdot\|_{H^s}$ for norms of $\eta$ on $\Sigma$.
\end{enumerate}
\subsection{Organization of the rest of this paper.}
In section 2, we develop the machinery of basic linear energy-dissipation as our start point of the nonlinear energy method. 
In section 3, we give the elliptic analysis based on an bootstrap argument only depending on the surface function $\eta$ in low regularity. 
In section 4, we construct two ingredients for our \textit{a priori estimates}, that are horizontal energy-dissipation estimates and elliptic estimates. 
In section 5, we construct the \textit{a priori estimates}, then develop the ideas of local well-posedness and complete the main result.

\section{Weak formulation and energy}
In this section, we develop the machinery of basic linear energy-dissipation, this is our start point of the nonlinear energy method.

Suppose that $u$ and $\eta$ are given, as well as $\mathcal{A}$, $\mathcal{N}$, etc., in terms of $\eta$. Then we consider the following system for triples $(v,q,\zeta)$:
\beq
\left\{
\begin{aligned}
&\pa_tv-\pa_t\bar{\eta} WK\pa_3v+u\cdot\nabla_{\mathcal{A}} v + \dive_{\mathcal{A}}S_{\mathcal{A}}(q,v)= F^1\quad&\text{in}&\ \Om,\\
&\dive_{\mathcal{A}}v =F^2\quad&\text{in}&\ \Om,\\
&S_{\mathcal{A}}(q,v) \mathcal{N} =\zeta\mathcal{N}-\Delta_\ast\zeta\mathcal{N}+F^3\quad&\text{on}&\ \Sigma,\\
&\pa_t\zeta = v\cdot\mathcal{N} +F^4\quad&\text{on}&\ \Sigma,\\
&v = 0\quad&\text{on}&\ \Sigma_b.\label{eq:new_NS}
\end{aligned}
\right.
\eeq
Here $(v,q,\zeta)$ represents the one temporal derivative of $(u,p,\eta)$ or the horizontal spatial derivatives of $(u,p,\eta)$ up to second order. 
The following lemma is the start point of horizontal energy-dissipation estimates.
\begin{lemma}\label{lem:energy}
Suppose that $u$ and $\eta$ are given and $(v,q,\zeta)$ are sufficiently smooth satisfying \eqref{eq:new_NS}. Then
\beq\label{eq:energy_1}
\begin{aligned}
 &\f12\f{d}{dt} \left( \|v\|_{\mathcal{H}^0}^2+ \|\zeta\|_{H^1(\Sigma)}^2 \right)+\f12 \|v\|_{\mathcal{H}^1}^2
= \left(F^1,v \right)_{\mathcal{H}^0}+ \left(q,F^2 \right)_{\mathcal{H}^0} - \int_\Sigma v\cdot F^3
+\int_{\Sigma}\left(\zeta-\Delta_\ast\zeta \right)F^4,
\end{aligned}
\eeq
where the space $\mathcal{H}^0$ and $\mathcal{H}^1$ are time-dependent spaces defined by
\[
   \mathcal{H}^0(t)= \left\{u:\Om\to\mathbb{R}^3 |  \sqrt{J}u\in L^2(\Om) \right\}
\text{ and }
  \mathcal{H}^1(t)= \left\{u: \int_{\Om} \left|\mathbb{D}_{\mathcal{A}(t)}u \right|^2J<\infty \right\},
  \]
  respectively.
 \end{lemma}
\begin{proof}
  We take the $L^2$-inner product of the first equation in \eqref{eq:new_NS} with $vJ$ over $\Om$ and get 
\beq
  \begin{aligned}
    \int_\Om  \left(\pa_tv-\pa_t\bar{\eta} WK\pa_3v+
u\cdot\nabla_{\mathcal{A}}v \right) \cdot vJ  + \int_\Om \dive_{\mathcal{A}}S_{\mathcal{A}}(q,v)\cdot vJ
    :=I+II=\int_\Om F^1\cdot vJ.
  \end{aligned}
\eeq
A simple computation by changing of variables gives 
\beq\label{eq:e_1}
I=\int_\Om \left(\pa_tv-\pa_t\bar{\eta} WK\pa_3v +  u\cdot\nabla_{\mathcal{A}}v \right)\cdot vJ = \f12\f{d}{dt}\int_\Om |v|^2J.
  \eeq
Then by integration by parts,
  \beq\label{eq:e_2}
\begin{aligned}
II =\int_\Om \dive_{\mathcal{A}}S_{\mathcal{A}}(q,v)\cdot vJ
&=\int_\Om-q\dive_{\mathcal{A}}vJ+\f12 \left|\mathbb{D}_{\mathcal{A}}v \right|^2J + \int_\Sigma S_{\mathcal{A}}(q,v)\mathcal{N}\cdot v\\
&=\int_\Om-q\dive_{\mathcal{A}}vJ+\f12 \left|\mathbb{D}_{\mathcal{A}}v \right|^2J
+\int_\Sigma \left(\zeta-\Delta_\ast\zeta \right)\left(v\cdot\mathcal{N} \right) + v\cdot F^3.
\end{aligned}
  \eeq
 Then another integration by parts shows 
\beq\label{eq:e_3}
 \begin{aligned}
\int_\Sigma \left(\zeta-\Delta_\ast\zeta \right) \left(v\cdot\mathcal{N} \right)
&=\int_\Sigma \left(\zeta-\Delta_\ast\zeta \right) \left(\pa_t\zeta-F^4 \right)
=\int_\Sigma\zeta\pa_t\zeta+\nabla_\ast\zeta \pa_t\nabla_\ast\zeta -  \left(\zeta-\Delta_\ast\zeta \right)F^4\\
&=\f12\f{d}{dt}\int_\Sigma  \left|\zeta \right|^2+ \left|\nabla_\ast\zeta \right|^2-\int_\Sigma \left(\zeta-\Delta_\ast\zeta \right)F^4.
\end{aligned}
  \eeq
Consequently, \eqref{eq:energy_1} holds after combining \eqref{eq:e_1}-\eqref{eq:e_3}.
\end{proof}
\begin{remark}
The space $\mathcal{H}^0$ is endowed with the norm $\left(\int_\Om   |u|^2J\right)^{1/2}$, that is comparable with the usual $L^2(\Om)$ norm,
and the space $\mathcal{H}^1$ is equivalent to the usual $H^1(\Om)$. The equivalence is referred to \cite[Lemma 2.1]{GT1} for small data of $\eta$,
and \cite[Lemma 2.9]{Wu} for large data of $\eta$. The $\mathcal{H}^0$ is introduced naturally, 
since after a change of variables, the $L^2(\Om(t))$-norm is reduced to $\mathcal{H}^0$-norm.
\end{remark}
\section{Elliptic estimates}
In this section, we give the elliptic analysis based on the bootstrap argument in low regularity Sobolev spaces. 
The results are important in section \ref{se4} and section \ref{se5}.

 Suppose $\eta$, $\mathcal{A}$, $\mathcal{N}$, etc., are given. We consider the following elliptic systems
 \beq\label{eq:Stokes_1}
 \left\{
 \begin{aligned}
&\dive_{\mathcal{A}}S_{\mathcal{A}}(q,v)=  G^1\quad&\text{in}&\ \Om,\\
&\dive_{\mathcal{A}}v=G^2\quad&\text{in}&\ \Om,\\
 &S_{\mathcal{A}}(q,v)\mathcal{N}=\xi\mathcal{N}-\Delta_\ast\xi\mathcal{N}+G^3\quad&\text{on}&\ \Sigma,\\
 &v\cdot\mathcal{N}=G^4\quad&\text{on}&\ \Sigma,\\
 &v=0\quad&\text{on}&\ \Sigma_b,
 \end{aligned}
 \right.
 \eeq
 where $\Delta_\ast=\pa_1^2+\pa_2^2$ is the horizontal Laplace operator.

In order to solve \eqref{eq:Stokes_1}, we first consider the case $\xi=\eta=0$, 
so that \eqref{eq:Stokes_1} is reduced to the usual Stokes system with constant coefficient and we will handle $(v,q)$ decoupling with $\xi$ by the fixed point theory.
We will treat \eqref{eq:Stokes_1} as a perturbation of the usual Stokes system and use the bootstrap argument to obtain the elliptic estimates. 
The key point of the bootstrap argument is the equation $\dive_{\mathcal{A}}v=G^2$, since this ensures that the normal part of $v$ could be represented by the horizontal components. If we suppose that $\|\eta\|_3^2\le \delta$ for a sufficiently small $\delta$,
then by the Sobolev embedded theory for $\left\|\nabla_\ast\eta \right\|_{L^\infty(\Sigma)}^2+ \left\|\nabla\bar{\eta} \right\|_{L^\infty(\Om)}^2 \lesssim \|\eta\|_3^2$, we have
 \beq\label{est:bound_coeffi}
\left\|\mathcal{N} \right\|_{L^\infty(\Sigma)}^2 + \left\|\mathcal{A} \right\|_{L^\infty(\Om)}^2\lesssim 1.
 \eeq
 \subsection{Stokes system}
 We first consider the following Stokes system of constant coefficients
 \beq\label{eq:stokes_2}
 \left\{
 \begin{aligned}
   &\dive S(q,v)=R^1\quad&\text{in}&\ \Om,\\
   &\dive v=R^2\quad&\text{in}&\ \Om,\\
   &\mathbb{D}v\nu\cdot\tau=R^3 \quad&\text{on}&\ \Sigma,\\
   &v\cdot\nu=R^4\quad&\text{on}&\ \Sigma,\\
   &v=0\quad&\text{on}&\ \Sigma_b,
 \end{aligned}
 \right.
 \eeq
 where $\nu$ and $\tau$ are the unit normal and tangential of $\Sigma$, respectively.
 We now formulate a definition of weak solutions to \eqref{eq:stokes_2}.
\begin{definition}\label{def:weak}
Suppose $R^1\in (H^1)^\ast(\Om)$, $R^2\in H^0(\Om)$, $R^3\in H^{-1/2}(\Sigma)$ and $R^4\in H^{1/2}(\Sigma)$ satisfying
   \beq
   \int_\Om R^2=\int_\Sigma R^4.
   \eeq
A pair $(v,q)\in H^1(\Om)\times H^0(\Om)$ is called a weak solution of \eqref{eq:stokes_2}, provided that $\dive v=R^2$, 
$v\cdot\nu=R^4$ on $\Sigma$, $v=0$ on $\Sigma_b$, and
   \beq
   \f12(v,w)_1- \left(q,\dive w \right)_0=\left<R^1,w\right>_{\ast}+  \left<R^3, w\cdot\tau\right>_{H^{-1/2}(\Sigma)\times H^{1/2}(\Sigma)}
   \eeq
 for any $w\in\mathscr{W}= \left\{w\in H^1(\Om) |w\cdot\nu=0\ \text{on}\ \Sigma,\ w=0\ \text{on}\ \Sigma_b \right\}$.
\end{definition}
 Now we establish the existence and uniqueness of weak solutions for \eqref{eq:Stokes_1}.
 \begin{Lemma}
Let $\left(R^1,R^2,R^3,R^4 \right)$ be as in Definition \ref{def:weak},
there exists a unique pair $(v,q)\in H^1(\Om)\times H^0(\Om)$ as a weak solution to \eqref{eq:Stokes_1}.
 Furthermore, 
   \beq
\|v\|_1^2+\|q\|_0^2\lesssim  \left\|R^1 \right\|_{\ast}^2+ \left\|R^2 \right\|_0^2+ \left\|R^3 \right\|_{-1/2}^2+ \left\|R^4 \right\|_{1/2}^2.
   \eeq
 \end{Lemma}
 \begin{proof}
We first consider 
   \beq
   \left\{
   \begin{aligned}
     &\Delta\varphi=R^2\quad&\text{in}&\ \Om,\\
     &\nabla\varphi\cdot\nu=R^4\quad&\text{on}&\ \Sigma,\\
     &\nabla\varphi\cdot\nu=0\quad&\text{on}&\ \Sigma_b.
   \end{aligned}
   \right.
   \eeq
From the well-known theory of elliptic equations subjecting to Neumann boundary conditions, 
there exists $\varphi\in H^2(\Om)$ satisfying
   \beq
 \left\|\nabla\varphi \right\|_1^2\lesssim  \left\|R^2 \right\|_0^2+ \left\|R^4 \right\|_{1/2}^2.
   \eeq
Then, as \cite[Lemma 4.2]{Beale1}, there exists a vector-valued function $\psi\in H^2(\Om)$ satisfying
   \[
   \left\{
   \begin{aligned}
   &\psi=0,\   \nabla\psi\cdot\nu=\nu\times  \left(- \pa_1\varphi, - \pa_2\varphi, 0 \right)^T\quad&\text{on}&\ \Sigma_b,\\
   &\psi=0 \  &\text{on}&\ \Sigma
   \end{aligned}
   \right.
   \]
such that
   \beq
   \left\{
\begin{aligned}
     &-\nabla\cdot \left(\nabla\times\psi \right)=0\quad&\text{in}&\ \Om,\\
     & \left(\nabla\times\psi \right)\cdot\nu=0\quad&\text{on}&\ \Sigma,\\
     &\nabla\times\psi= - \left(\nabla_\ast\varphi, 0 \right)^T\quad&\text{on}&\ \Sigma_b.
   \end{aligned}
   \right.
   \eeq
 Combining the extension from the boundary $\pa\Om$ to $\Om$ (The similar method is also used in \cite[page 25-27]{La}), it satisfies 
\[
\left\|\nabla\times\psi \right\|_{H^1}^2\lesssim  \left\|\nabla\varphi \right\|_1^2\lesssim  \left\|R^2 \right\|_0^2+ \left\|R^4 \right\|_{1/2}^2.
   \]
So we can derive
\beq
\left\{
\begin{aligned}
&\dive \left(v-\nabla\varphi-\nabla\times\psi \right)=0\quad&\text{in}&\ \Om,\\
& \left(v-\nabla\varphi-\nabla\times\psi \right)\cdot\nu=0\quad&\text{on}&\ \Sigma,\\
&v-\nabla\varphi-\nabla\times\psi=0  \quad&\text{on}&\ \Sigma_b.
\end{aligned}
\right.
\eeq
Then we switch the unknowns to $\bar{v}=v-\nabla\varphi-\nabla\times\psi$, and then restrict the test functions to $w\in{}_0^0H^1_\sigma(\Om)$ so that $\dive \bar{v}=0$ and the pressureless weak formulation for $v$ is
   \beq\label{eq:weak_1}
   \f12\left(\bar{v},w \right)_1=-\f12 \left(\nabla\varphi+\nabla\times\psi,w \right)_1+ \left<R^1,w\right>_{\ast}+\left<R^3, w\cdot\tau\right>_{H^{-1/2}(\Sigma)\times H^{1/2}(\Sigma)}
   \eeq
for any $w\in{}_0^0H^1_\sigma(\Om)$. Here the space ${}_0^0H^1_\sigma(\Om)$ is defined as
    \[
{}_0^0H^1_\sigma(\Om)= \left\{u\in H^1(\Om):  \dive u=0\ \text{in}\ \Om,\ w\cdot\nu=0\ \text{on}\ \Sigma,\ w=0\ \text{on}\ \Sigma_b \right\}.
    \]
Then the Riesz representation theorem provides a unique $\bar{v}\in{}_0^0H^1_\sigma(\Om)$ satisfying \eqref{eq:weak_1} and
   \[
   \begin{aligned}
   \|\bar{v}\|_1^2&\lesssim  \|\varphi\|_2^2+  \|\psi\|_2^2+ \left\|R^1 \right\|_{\ast}^2+ \left\|R^3 \right\|_{-1/2}^2
   \lesssim  \left\|R^1 \right\|_{\ast}^2+ \left\|R^2 \right\|_0^2+ \left\|R^3 \right\|_{-1/2}^2+ \left\|R^4 \right\|_{1/2}^2.
   \end{aligned}
   \]
To introduce the pressure $q$, we define a linear functional on $(\mathscr{W})^\ast$ as the difference of left-hand side and the right-hand side of \eqref{eq:weak_1}.
Then treating pressure as the Lagrangian of Stokes operator(see for instance \cite{Te}), there exists a unique $q\in H^0(\Om)$ (possibly up to a constant) satisfying
\beq
\|q\|_0^2\lesssim  \left\|R^1 \right\|_{\ast}^2+ \left\|R^2 \right\|_0^2+ \left\|R^3 \right\|_{-1/2}^2+ \left\|R^4 \right\|_{1/2}^2.
\eeq
 \end{proof}
Now we develop the second-order regularity of \eqref{eq:stokes_2}.
\begin{proposition}\label{prop:elliptic}
   Suppose $R^1\in H^0(\Om)$, $R^2\in H^1(\Om)$, $R^3\in H^{1/2}(\Sigma)$ and $R^4\in H^{3/2}(\Sigma_b)$ satisfying
   \beq
   \int_\Om R^2=\int_\Sigma R^4.
   \eeq
Then the problem \eqref{eq:stokes_2} admits a unique strong solution $(v,q)\in H^2(\Om)\times H^1(\Om)$ and 
\beq\label{est:stokes_1}
   \|v\|_2^2+\|q\|_1^2  \lesssim  \left\|R^1 \right\|_0^2+ \left\|R^2 \right\|_1^2+ \left\|R^3 \right\|_{1/2}^2+ \left\|R^4 \right\|_{3/2}^2.
   \eeq
\end{proposition}
\begin{proof}
Since the Stokes system \eqref{eq:stokes_2} is elliptic, following similar arguments in \cite{Te}, we can get the result.
Furthermore, \cite[Theorem 10.5]{ADN} provides \eqref{est:stokes_1}.
 \end{proof}
Rewriting the Stokes system \eqref{eq:Stokes_1} as the following perturbation form:
 \beq\label{eq:stokes_3}
 \left\{
 \begin{aligned}
   &\dive S(q,v)=R^1\quad&\text{in}&\ \Om,\\
   &\dive v=R^2\quad&\text{in}&\ \Om,\\
   &\mathbb{D}v\nu\cdot\tau=R^3\quad&\text{on}&\ \Sigma,\\
   &v\cdot\nu=R^4\quad&\text{on}&\ \Sigma,\\
   &v=0\quad&\text{on}&\ \Sigma_b,
 \end{aligned}
 \right.
 \eeq
 and
 \beq\label{eq:surface_1}
 q-\f1{|\mathcal{N}|^2}\mathbb{D}_{\mathcal{A}}v\mathcal{N}\cdot\mathcal{N}
=\xi-\Delta_\ast\xi+\f1{|\mathcal{N}|^2}G^3\cdot\mathcal{N}\quad\text{on}\ \Sigma,
 \eeq
 where
 \begin{align*}
R^1 & =R^1(v,q)=G^1+\dive_{I-\mathcal{A}}S_{\mathcal{A}}(q,v)-\dive\mathbb{D}_{I-\mathcal{A}}v,   &    R^2=R^2(v)=G^2+\dive_{I-\mathcal{A}}v,  \\
R^3  & =R^3(v)=-G^3\cdot\mathcal{T}+\mathbb{D}_{I-\mathcal{A}}v\mathcal{N}\cdot\mathcal{T}
+\mathbb{D}v(\nu-\mathcal{N})\cdot\tau+\mathbb{D}v\mathcal{N}\cdot(\tau-\mathcal{T}),     & R^4=R^4(v)=G^4+v\cdot(\nu-\mathcal{N}),
 \end{align*}
 and $\mathcal{T}$ is the tangential of free surface.
\begin{theorem}\label{thm:elliptic}
Suppose $G^1\in H^0(\Om)$, $G^2\in H^1(\Om)$, $G^3\in H^{1/2}(\Sigma)$, $G^4\in H^{3/2}(\Sigma)$ and $\|\eta\|_{5/2}\le \delta$ for $\delta$ sufficiently small. Then there exists a unique triple $(v,q,\xi)\in H^2(\Om)\times H^1(\Om)\times  H^{5/2}(\Sigma)$ solving \eqref{eq:Stokes_1}. Moreover,
\beq\label{est:elliptic_1}
   \|v\|_2^2+\|q\|_1^2+\|\xi\|_{5/2}^2\lesssim  \left\|G^1 \right\|_0^2 
+ \left\|G^2 \right\|_1^2+ \left\|G^3 \right\|_{1/2}^2+ \left\|G^4 \right\|_{3/2}^2.
   \eeq
Furthermore, if $G^1\in H^1(\Om)$, $G^2\in H^2(\Om)$, $G^3\in H^{3/2}(\Sigma)$ and $G^4\in H^{5/2}(\Sigma)$. Then
\beq\label{est:elliptic_2}
\begin{aligned}
   \|v\|_3^2+\|q\|_2^2+\|\xi\|_{7/2}^2&\lesssim  \left\|G^1 \right\|_1^2+ \left\|G^2 \right\|_2^2 
+ \left\|G^3 \right\|_{3/2}^2+ \left\|G^4 \right\|_{5/2}^2
 +\|\eta\|_{7/2}^2 \left( \left\|G^1 \right\|_0^2+ \left\|G^2 \right\|_1^2 
+ \left\|G^3 \right\|_{1/2}^2+ \left\|G^4 \right\|_{3/2}^2 \right).
 \end{aligned}
 \eeq
\end{theorem}
\begin{proof}
We will solve \eqref{eq:stokes_3} and \eqref{eq:surface_1} by Banach's fixed point theorem. For any $(u,p)\in H^2(\Om)\times H^1(\Om)$, we have
\beq\label{est:bound}
\begin{aligned}
 & \left\|R^1(u,p) \right\|_0^2+ \left\|R^2(u) \right\|_1^2+ \left\|R^3(u) \right\|_{1/2}^2+ \left\|R^4(u) \right\|_{3/2}^2\\
\lesssim &  \left\|G^1 \right\|_0^2+ \left\|G^2 \right\|_1^2+ \left\|G^3 \right\|_{1/2}^2+ \left\|G^4 \right\|_{3/2}^2+P \left(\|\eta\|_{5/2}^2 \right) \left(\|u\|_2^2+\|p\|_1^2 \right),
\end{aligned}
\eeq
where $P(\cdot)$ is a polynomial satisfying $P(0)=0$. Now we consider the linear problem:
\beq\label{eq:stokes_4}
\left\{
\begin{aligned}
&\dive S(q,v)=R^1(u,p)\quad&\text{in}&\ \Om,\\
&\dive v=R^2(u)\quad&\text{in}&\ \Om,\\
&\mathbb{D}v\nu\cdot\tau=R^3(u)\quad&\text{on}&\ \Sigma,\\
&v\cdot\nu=R^4(u)\quad&\text{on}&\ \Sigma,\\
&v=0\quad&\text{on}&\ \Sigma_b.
\end{aligned}
\right.
\eeq
From Proposition \ref{prop:elliptic}, there exists a unique pair $(v,q)\in H^2(\Om)\times H^1(\Om)$ which solves \eqref{eq:stokes_4}.
Moreover, by \eqref{est:bound}, we have
\beq
\|v\|_2^2+\|q\|_1^2\lesssim  \left\|G^1 \right\|_0^2+ \left\|G^2 \right\|_1^2+ \left\|G^3 \right\|_{1/2}^2 
+ \left\|G^4 \right\|_{3/2}^2+P \left(\|\eta\|_{5/2}^2 \right) \left(\|u\|_2^2+\|p\|_1^2 \right).
\eeq
Similarly, suppose $ \left(\tilde{u},\tilde{p} \right)\in H^2(\Om)\times H^1(\Om)$ and $(\tilde{v},\tilde{q})$ solves
\beq\label{eq:stokes_5}
\left\{
\begin{aligned}
&\dive S(\tilde{q},\tilde{v}) =R^1(\tilde{u},\tilde{p})\quad&\text{in}&\ \Om,\\
&\dive \tilde{v}=R^2(\tilde{u})\quad&\text{in}&\ \Om,\\
&\mathbb{D}\tilde{v}\nu\cdot\tau=R^3(\tilde{u})\quad &\text{on}&\ \Sigma,\\ &\tilde{v}\cdot\nu=R^4(\tilde{u})\quad&\text{on}&\ \Sigma,\\
&\tilde{v}=0\quad&\text{on}&\ \Sigma_b.
\end{aligned}
\right.
\eeq
Subtracting \eqref{eq:stokes_5} from \eqref{eq:stokes_4}, we see $v-\tilde{v}$ satisfies
\beq\label{eq:stokes_6}
\left\{
\begin{aligned}
 &\dive S \left(q-\tilde{q},v-\tilde{v} \right)= R^1(u,p)-R^1\left(\tilde{u},\tilde{p} \right)\quad&\text{in}&\ \Om,\\
 &\dive \left(v-\tilde{v} \right)=R^2(u)-R^2(\tilde{u})\quad&\text{in}&\ \Om,\\
  &\mathbb{D}(v-\tilde{v})\nu\cdot\tau=R^3(u)-R^3 \left(\tilde{u} \right)\quad&\text{on}&\ \Sigma,\\
   &\left(v-\tilde{v} \right)\cdot\nu=R^4(u)-R^4\left(\tilde{u} \right)\quad&\text{on}&\ \Sigma,\\
  &v-\tilde{v} =0\quad&\text{on}&\ \Sigma_b.
 \end{aligned}
\right.
\eeq
Then by Proposition \ref{prop:elliptic}, we obtain 
\beq
\left\|v-\tilde{v} \right\|_2^2+ \left\|q-\tilde{q} \right\|_1^2\le P \left(\|\eta\|_{5/2}^2 \right) \left(\left\|u-\tilde{u} \right\|_2^2
+ \left\|p-\tilde{p} \right\|_1^2 \right).
\eeq
We can take a small constant $\delta$ such that $P\left(\|\eta\|_{5/2}^2 \right)<\f12$, provided $\|\eta\|_{5/2}\le\delta$.
Then by Banach's fixed point Theorem (see for instance \cite{Evans}), the mapping $(u,p)\mapsto(v,q)$ has a unique fixed point.
Therefore the Stokes problem \eqref{eq:stokes_3} has a unique solution $(v,q)\in H^2(\Om)\times H^1(\Om)$ which satisfies 
\beq\label{est:elliptic_3}
\|v\|_2^2+\|q\|_1^2\lesssim \left\|G^1\right\|_0^2+ \left\|G^2 \right\|_1^2+ \left\|G^3\right\|_{1/2}^2+ \left\|G^4\right\|_{3/2}^2.
\eeq
For any $\zeta\in H^1(\Sigma)$, by \eqref{eq:surface_1}, we have
\beq
\int_\Sigma \left(\xi\zeta+\nabla_\ast\xi\cdot\nabla_\ast\zeta  \right)=\int_\Sigma q\zeta-\f1{ \left|\mathcal{N} \right|^2}\mathbb{D}_{\mathcal{A}}v\mathcal{N}\cdot\mathcal{N}\zeta-\f1{ \left|\mathcal{N} \right|^2}G^3\cdot\mathcal{N}\zeta.
\eeq
By the H\"older inequality and Sobolev inequality, we have 
\beq
\begin{aligned}
&\int_\Sigma  \left(\xi\zeta+\nabla_\ast\xi\cdot\nabla_\ast\zeta \right)\lesssim \|\xi\|_{L^2}  \|\zeta\|_{L^2}
+\left\|\nabla_\ast\xi \right\|_{L^2} \left\|\nabla_\ast\zeta \right\|_{L^2}\lesssim\|\xi\|_1^2+\|\zeta\|_1^2.
\end{aligned}
\eeq
 Then the Lax-Milgram Theorem guarantees \eqref{eq:surface_1} has a unique weak solution $\xi\in H^1(\Sigma)$.
By the regularity theory of elliptic equations and trace theory, we have
   \beq\label{est:elliptic_4}
   \begin{aligned}
\|\xi\|_{5/2}^2&\lesssim \|\xi\|_{1/2}^2+ \left\|\nabla_\ast^2\xi \right\|_{1/2}^2\lesssim \|\xi\|_{1/2}^2+ \|q\|_{H^{1/2}(\Sigma)}^2
+\|v\|_{H^{3/2}(\Sigma)}^2+ \left\|G^3 \right\|_{1/2}^2\\
&\lesssim  \left\|G^1 \right\|_0^2+ \left\|G^2 \right\|_1^2+ \left\|G^3 \right\|_{1/2}^2+ \left\|G^4 \right\|_{3/2}^2.
\end{aligned}
   \eeq
We now use a bootstrap argument to enhance the regularity of solutions for \eqref{eq:Stokes_1}.
For $j=1,2$, we apply the horizontal derivative to \eqref{eq:Stokes_1} and get
   \beq\label{eq:stokes_7}
\left\{
\begin{aligned}
  &\dive_{\mathcal{A}}S_{\mathcal{A}} \left(\pa_jq,\pa_jv \right)=G^{1,1}\ &\text{in}&\ \Om,\\
  &\dive_{\mathcal{A}}\pa_jv=G^{2,1}\ &\text{in}&\ \Om,\\
  &S_{\mathcal{A}} \left(\pa_jq,\pa_jv \right)\mathcal{N}= \pa_j\xi\mathcal{N}-\Delta_\ast\pa_j\xi\mathcal{N} 
+ G^{3,1}\ &\text{on}&\ \Sigma,\\
  &\pa_jv\cdot\mathcal{N}= G^{4,1}\ &\text{on}&\ \Sigma,\\
  &\pa_jv=0\ &\text{on}&\ \Sigma_b,
\end{aligned}
\right.
\eeq
where
\beq
\begin{aligned}
G^{1,1}&=\pa_jG^1-\nabla_{\pa_j\mathcal{A}}q+\dive_{\mathcal{A}}\mathbb{D}_{\pa_j\mathcal{A}}v+\dive_{\pa_j\mathcal{A}}\mathbb{D}_{\mathcal{A}}v,
& G^{2,1}=\pa_jG^2-\dive_{\pa_j\mathcal{A}}v,\\
  G^{3,1}&=\pa_jG^3-S_{\mathcal{A}}(q,v)\pa_j\mathcal{N}+\mathbb{D}_{\pa_j\mathcal{A}}v\mathcal{N}+\xi\pa_j\mathcal{N}-\delta_\ast^2\xi\pa_j\mathcal{N},
& G^{4,1}  =\pa_jG^4-v\pa_j\mathcal{N}.
\end{aligned}
\eeq
Then \eqref{est:elliptic_1} guarantees
\beq\label{est:stokes_2}
\left\|\pa_jv \right\|_2^2+ \left\|\pa_jq \right\|_1^2+ \left\|\pa_j\xi \right\|_{5/2}^2\lesssim \left\|G^{1,1} \right\|_0^2+ \left\|G^{2,1} \right\|_1^2+ \left\|G^{3,1} \right\|_{1/2}^2+ \left\|G^{4,1} \right\|_{3/2}^2.
\eeq
Direct calculation for the estimates of forcing terms $G^{k, 1}$, $k=1, 2, 3, 4$, by utilizing the Sobolev inequalities, the H\"older inequalities and interpolation estimate, are as follows. For $\varepsilon$ sufficiently small and any $3/2<s<2$, we have
\beq\label{est:f_1}
\begin{aligned}
\left\|G^{1,1} \right\|_0^2&\lesssim  \left\|\pa_jG^1 \right\|_0^2
+ \left\|\pa_j\mathcal{A} \right\|_{L^4}^2 \left\|\nabla q \right\|_{L^4}^2 + \left\|\mathcal{A} \right\|_{L^\infty}^2 \left( \left\|\pa_j\mathcal{A} \right\|_{L^4}^2 \left\|\nabla^2v \right\|_{L^4}^2
+ \left\|\nabla\pa_j\mathcal{A} \right\|_{L^2}^2 \left\|\nabla v \right\|_{L^\infty}^2 \right)\\
&\quad+ \left\|\pa_j\mathcal{A} \right\|_{L^4}^2 \left( \left\|\mathcal{A} \right\|_{L^\infty}^2 \left\|\nabla^2v \right\|_{L^2}^2
+ \left\|\nabla\mathcal{A} \right\|_{L^4}^2 \left\|\nabla v \right\|_{L^\infty}^2 \right)\\
&\lesssim   \left\|\pa_jG^1 \right\|_0^2+ \left\|\eta \right\|_{5/2}^2 \left(\|v\|_2^2
+ \left\|\nabla v \right\|_s^2+ \left\|\nabla q \right\|_{L^2}^{1/2}
\left\|\nabla^2 q \right\|_{L^2}^{3/2}+ \left\|\nabla^2v \right\|_{L^2}^{1/2} \left\|\nabla^3v \right\|_{L^2}^{3/2} \right)\\
&\lesssim   \left\|\pa_jG^1 \right\|_0^2+ \|\eta\|_{5/2}^2 \left(\|v\|_2^2+ \left\|\nabla^2 v \right\|_{L^2}^{2(2-s)} \left\|\nabla^3v \right\|_{L^2}^{2(s-1)}
+ \|\nabla q\|_{L^2}^{1/2} \left\|\nabla^2 q \right\|_{L^2}^{3/2}+ \left\|\nabla^2v \right\|_{L^2}^{1/2} \left\|\nabla^3v \right\|_{L^2}^{3/2} \right)\\
&\lesssim \left\|\pa_jG^1 \right\|_0^2+   \|\eta\|_{5/2}^2 \left(\|v\|_2^2+\|q\|_1^2 \right)
+\varepsilon \left( \left\|\nabla^3v \right\|_{L^2}^2+ \left\|\nabla^2 q \right\|_{L^2}^2 \right).
\end{aligned}
\eeq
Similarly, we have
\beq\label{est:f_2}
\begin{aligned}
\left\|G^{2,1} \right\|_1^2&\lesssim  \left\|\pa_jG^2 \right\|_1^2
+ \left\|\pa_j\mathcal{A} \right\|_{L^4}^2
 \left( \|\nabla v\|_{L^4}^2+ \left\|\nabla^2 v \right\|_{L^4}^2 \right)+ \left\|\nabla\pa_j\mathcal{A}\right\|_{L^2}^2 \left\|\nabla v \right\|_{L^\infty}^2\\
&\lesssim  \left\|\pa_jG^2 \right\|_1^2 +\|\eta\|_{5/2}^2 \|v\|_2^2 +\varepsilon \left\|\nabla^3v \right\|_{L^2}^2.
\end{aligned}
\eeq
Then by the trace embedding and extension theory in Sobolev spaces along with \eqref{est:bound_coeffi}, we have
\beq\label{est:f_31}
\begin{aligned}
\left\|S_{\mathcal{A}}(q,v)\pa_j\mathcal{N} \right\|_{H^{1/2}(\Sigma)}^2&\lesssim \left\|S_{\mathcal{A}}(q,v)\pa_j\nabla_\ast\bar{\eta} \right\|_{H^1(\Om)}^2\\
&\lesssim  \left( \|q\|_{s}^2+ \left\|\mathcal{A} \right\|_{L^\infty}^2 \|\nabla v\|_{s}^2 \right) \left\|\nabla\pa_j\nabla_\ast\bar{\eta} \right\|_{L^2}^2+ \left\|\pa_j\nabla_\ast\bar{\eta} \right\|_{L^4}^2 \left(\|q\|_{L^4}^2+ \|\nabla q\|_{L^4}^2 \right)\\
&\quad+ \left\|\pa_j\nabla_\ast\bar{\eta} \right\|_{L^4}^2 \left\|\mathcal{A} \right\|_{L^\infty}^2 \left(  \|\nabla v\|_{L^4}^2+ \left\|\nabla^2 v \right\|_{L^4}^2 \right)+ \left\|\pa_j\nabla_\ast\bar{\eta} \right\|_{L^6}^2 \left\|\nabla\mathcal{A} \right\|_{L^6}^2\|\nabla v\|_{L^6}^2\\
 &\lesssim \|\eta\|_{5/2}^2 \left(\|v\|_2^2+\|q\|_1^2 \right)+\varepsilon \left( \left\|\nabla^3v \right\|_{L^2}^2+ \left\|\nabla^2 q \right\|_{L^2}^2 \right),
\end{aligned}
\eeq
and
\beq\label{est:f_32}
\begin{aligned}
\left\|\mathbb{D}_{\pa_j\mathcal{A}}v\mathcal{N} \right\|_{H^{1/2}(\Sigma)}^2&\lesssim \left\|\mathbb{D}_{\pa_j\mathcal{A}}v \right\|_1^2+ \left\|\mathbb{D}_{\pa_j\mathcal{A}}v\nabla_\ast\eta \right\|_{H^{1/2}(\Sigma)}^2\\
&\lesssim \left\|\pa_j\mathcal{A} \right\|_{L^4}^2 \left( \|\nabla v\|_{L^4}^2+ \left\|\nabla^2v \right\|_{L^4}^2 \right)+ \left(1+\|\nabla_\ast\bar{\eta}\|_{L^\infty}^2 \right) \left\|\nabla\pa_j\mathcal{A} \right\|_{L^2}^2 \|\nabla v\|_{L^\infty}^2\\
&\quad+ \left\|\nabla_\ast\bar{\eta} \right\|_{L^\infty}^2 \left\|\pa_j\mathcal{A} \right\|_{L^4}^2 \left\|\nabla^2v \right\|_{L^4}^2+ \left\|\nabla\nabla_\ast\bar{\eta} \right\|_{L^6}^2 \left\|\pa_j\mathcal{A} \right\|_{L^6}^2 \left\|\nabla v \right\|_{L^6}^2\\
&\lesssim \|\eta\|_{5/2}^2 \|v\|_2^2+\varepsilon \left\|\nabla^3v \right\|_{L^2}^2.
\end{aligned}
\eeq
Combining \eqref{est:f_31} and \eqref{est:f_32}, we have
\beq\label{est:f_3}
\begin{aligned}
\left\|G^{3,1}\right\|_{1/2}^2&\lesssim  \left\|\pa_jG^3 \right\|_{1/2}^2+ \|\eta\|_{5/2}^2 \left(\|v\|_2^2+\|q\|_1^2+\|\xi\|_{7/2}^2+ \left\|G^5 \right\|_{5/2}^2 \right) +\varepsilon \left( \left\|\nabla^3v \right\|_{L^2}^2+ \left\|\nabla^2 q \right\|_{L^2}^2 \right)\\
&\lesssim  \left\|\pa_jG^3 \right\|_{1/2}^2+  \|\eta\|_{5/2}^2 \left(\|v\|_2^2+\|q\|_1^2+ \|\eta\|_{5/2}^2+ \left\|G^5 \right\|_{5/2}^2 \right) +\varepsilon \left( \left\|\nabla^3v \right\|_{L^2}^2+ \left\|\nabla^2 q \right\|_{L^2}^2+ \|\xi\|_{7/2}^2 \right).
\end{aligned}
\eeq
For the term $G^{4,1}$, we have 
\beq\label{est:f_4}
\begin{aligned}
  \left\|G^{4,1} \right\|_{3/2}^2&\lesssim  \left\|\pa_jG^4 \right\|_{3/2}^2+ \|v\|_2^2 \|\eta\|_{7/2}^2.
\end{aligned}
\eeq
Then \eqref{est:elliptic_1}, \eqref{est:stokes_2}, \eqref{est:f_2}, \eqref{est:f_3} and \eqref{est:f_4} imply for $j=1, 2$,
\beq
\begin{aligned}
& \left\|\pa_jv \right\|_2^2+ \left\| \pa_jq \right\|_1^2
+ \left\|\pa_j\xi \right\|_{5/2}^2 \\
\le &  C \left(1+\|\eta\|_{7/2}^2 \right) \left( \left\|G^1 \right\|_0^2+ \left\|G^2 \right\|_1^2
+ \left\|G^3 \right\|_{1/2}^2+ \left\|G^4 \right\|_{3/2}^2 \right)\\
& +C \left( \left\|\pa_jG^1 \right\|_0^2+ \left\| \pa_jG^2 \right\|_1^2+ \left\|\pa_jG^3 \right\|_{1/2}^2+ \left\|\pa_jG^4 \right\|_{3/2}^2 \right)
+C\varepsilon \left( \left\|\nabla^3v \right\|_{L^2}^2 + \left\|\nabla^2 q \right\|_{L^2}^2 + \|\xi\|_{7/2}^2 \right).
\end{aligned}
\eeq
Finally, we turn to the estimate of $\pa_3v$ and $\pa_3q$. From the three components of $\dive_{\mathcal{A}}S_{\mathcal{A}}(q,v)=G^1$,
we have
\[
\begin{aligned}
&  -\left(\pa_{11} v_1 + A\pa_{11} v_3  \right) - \left(\pa_{22} v_1 + A\pa_{22} v_3 \right)
 - \left(1 + A^2 + B^2 \right)K^2 \left(\pa_{33} v_1 + A\pa_{33} v_3 \right)
+2AK \left(\pa_{13} v_1 + A\pa_{13} v_3 \right)\\
& +2BK \left(\pa_{23} v_1 + A\pa_{23} v_3 \right)
- \left( AK\pa_3 (AK) + BK\pa_3 (BK)- \pa_1 (AK)-\pa_2 (BK)
+ K\pa_3 K \right) \left(\pa_3 v_1 + A\pa_3 v_3 \right)
+ \pa_1 q \\
=  &  G^1_1-\pa_1G^2 + AG^1_3,
\end{aligned}
\]
which contains the singular term $\left(1 + A^2 + B^2 \right)K^2 \left(\pa_{33} v_1 + A\pa_{33} v_3 \right)$. Applying $\pa_3$ on both sides of the above equation, the term $ \left(1 + A^2 + B^2 \right)K^2 \left(\pa_3^3 v_1 + A\pa_3^3 v_3 \right)$ implies
\beq\label{est:u_133}
\begin{aligned}
\left\|\pa_3^3 v_1 + A\pa_3^3 v_3 \right\|_0^2
& \lesssim \left( 1+ \|\eta\|_{7/2}^2 \right) \left( \left\|G^1 \right\|_0^2+ \left\|G^2 \right\|_1^2
+ \left\|G^3 \right\|_{1/2}^2+ \left\|G^4 \right\|_{3/2}^2 \right) + \left\|\nabla G^1 \right\|_0^2+ \left\|\nabla G^2 \right\|_1^2  \\
&\quad + \left\|\nabla_\ast G^3 \right\|_{1/2}^2+ \left\|\nabla_\ast G^4 \right\|_{3/2}^2
+\varepsilon \left( \left\|\nabla^3v \right\|_{L^2}^2+ \left\|\nabla^2 q \right\|_{L^2}^2+\|\xi\|_{7/2}^2 \right).
\end{aligned}
\eeq
Similarly, we have
\beq\label{est:u_233}
\begin{aligned}
\left\|\pa_3^3 v_2 + B\pa_3^3 v_3 \right\|_0^2&\lesssim  \left(1+\|\eta\|_{7/2}^2 \right) \left( \left\|G^1 \right\|_0^2+ \left\|G^2 \right\|_1^2+ \left\|G^3 \right\|_{1/2}^2+ \left\|G^4 \right\|_{3/2}^2 \right) + \left\|\nabla G^1 \right\|_0^2+ \left\|\nabla G^2 \right\|_1^2 \\
&\quad+ \left\|\nabla_\ast G^3 \right\|_{1/2}^2+ \left\|\nabla_\ast G^4 \right\|_{3/2}^2
+\varepsilon \left( \left\|\nabla^3v \right\|_{L^2}^2+ \left\|\nabla^2 q \right\|_{L^2}^2+\|\xi\|_{7/2}^2 \right).
\end{aligned}
\eeq
Since the divergence equation $\dive_{\mathcal{A}}v=G^2$ is equivalent to
\[
K \left(1 + A^2 + B^2 \right)\pa_3 v_3
= G^2-\pa_1 v_1- \pa_2 v_2 + AK \left(\pa_3 v_1 + A\pa_3 v_3 \right) + BK \left(\pa_3 v_2 + B\pa_3 v_3 \right),
\]
we can take $\pa_3^2$ on both sides of the above equation and get 
\beq \label{est:u_333}
\begin{aligned}
\left\|\pa_3^3 v_3 \right\|_0^2
& \lesssim  \left( 1 + \|\eta\|_{7/2}^2 \right) \left( \left\|G^1 \right\|_0^2+ \left\|G^2 \right\|_1^2+ \left\|G^3 \right\|_{1/2}^2
+ \left\|G^4 \right\|_{3/2}^2 \right)
+ \left\|\nabla G^1 \right\|_0^2
\\
& \quad + \left\|\nabla G^2 \right\|_1^2+ \left\|\nabla_\ast G^3 \right\|_{1/2}^2 + \left\|\nabla_\ast G^4 \right\|_{3/2}^2
+\varepsilon \left( \left\|\nabla^3v \right\|_{L^2}^2+ \left\|\nabla^2 q \right\|_{L^2}^2+\|\xi\|_{7/2}^2 \right).
\end{aligned}
\eeq
Thus, \eqref{est:u_133}--\eqref{est:u_333} imply
\beq
\begin{aligned}
 \left\|\pa_3^3 v \right\|_0^2&\lesssim  \left(1+\|\eta\|_{7/2}^2 \right) \left( \left\|G^1 \right\|_0^2+ \left\|G^2 \right\|_1^2
 + \left\|G^3 \right\|_{1/2}^2+ \left\|G^4 \right\|_{3/2}^2 \right)
 + \left\|\nabla G^1 \right\|_0^2\\
&\quad + \left\|\nabla G^2 \right\|_1^2+ \left\|\nabla_\ast G^3 \right\|_{1/2}^2
 + \left\|\nabla_\ast G^4 \right\|_{3/2}^2
+\varepsilon \left( \left\|\nabla^3v \right\|_{L^2}^2
+ \left\|\nabla^2 q \right\|_{L^2}^2+ \|\xi\|_{7/2}^2 \right).
\end{aligned}
\eeq
Then we take $\pa_3$ on both sides of the third component of $\dive_{\mathcal{A}}S_{\mathcal{A}}(q,v)=G^1$ and obtain
\beq
\begin{aligned}
\left\|\pa_3^2 q \right\|_0^2&\lesssim  \left(1+\|\eta\|_{7/2}^2 \right) \left( \left\|G^1 \right\|_0^2
+ \left\|G^2 \right\|_1^2+ \left\|G^3 \right\|_{1/2}^2+ \left\|G^4 \right\|_{3/2}^2 \right)
+ \left\|\nabla G^1 \right\|_0^2\\
&\quad + \left\|\nabla G^2 \right\|_1^2  + \left\|\nabla_\ast G^3 \right\|_{1/2}^2
+ \left\|\nabla_\ast G^4 \right\|_{3/2}^2
+\varepsilon \left( \left\|\nabla^3v \right\|_{L^2}^2+ \left\|\nabla^2 q \right\|_{L^2}^2+\|\xi\|_{7/2}^2 \right).
\end{aligned}
\eeq
Hence
\[
\begin{aligned}
\|v\|_3^2+\| q\|_2^2&\lesssim  \left(1+\|\eta\|_{7/2}^2 \right) \left( \left\|G^1 \right\|_0^2 
+ \left\|G^2 \right\|_1^2+ \left\|G^3 \right\|_{1/2}^2+ \left\|G^4 \right\|_{1/2}^2 \right) 
+ \left\|\nabla G^1 \right\|_0^2+ \left\|\nabla G^2 \right\|_1^2 \\
&\quad + \left\|\nabla_\ast G^3 \right\|_{1/2}^2+ \left\|\nabla_\ast G^4\right\|_{1/2}^2 
+\varepsilon \left( \left\|\nabla^3v \right\|_{L^2}^2+ \left\|\nabla^2 q \right\|_{L^2}^2+\|\xi\|_{7/2}^2 \right),
\end{aligned}
\]
taking $\varepsilon$ sufficiently small, we get \eqref{est:elliptic_2}.
\end{proof}
\begin{remark}
Actually, the elliptic estimates \eqref{est:elliptic_1} and \eqref{est:elliptic_2} can be generalized under the assumption
$\left\|\eta-\eta_0 \right\|_{5/2}^2\le \delta$ instead of $\|\eta\|_{5/2}^2\le \delta$,
which could be done by employing the argument that $\mathcal{A}$ is a perturbation of $\mathcal{A}(0)$.
This is crucial to deal with the local well-posedness with large initial data.
\end{remark}

\section{Energy and dissipation estimates}\label{se4}
In this section, we first establish the horizontal energy-dissipation estimates, that is Theorem \ref{thm:horizontal}. We also give elliptic estimates in Theorem \ref{thm:enhanced_dissipation}.
The two estimates are two ingredients for our \textit{a priori estimates} in the section \ref{se5}.

\subsection{The horizontal energy--dissipation estimate}
 In order to close our \textit{a priori estimates} for \eqref{eq:new_NS}, we need to consider the unknowns up to twice order derivatives, thus we have to confirm the forcing terms on the right-hand side of \eqref{eq:new_NS} with respect to derivatives of order 0, 1. Actually, if we set the horizontal energy and dissipation as
\begin{align*}
\mathcal{E}_\shortparallel & =\int_\Om  \left(|u|^2+ \left|\pa_tu \right|^2
+ \left|\nabla_\ast u \right|^2+ \left|\nabla_\ast^2 u \right|^2 \right)J+ \left\|\pa_t\eta \right\|_{1,\Sigma}^2+  \|\eta\|_{3,\Sigma}^2,
\intertext{  and }
 \mathcal{D}_\shortparallel & =\int_\Om  \left( \left|\mathbb{D}_{\mathcal{A}}u \right|^2
 +\left|\mathbb{D}_{\mathcal{A}}\pa_tu \right|^2+ \left|\mathbb{D}_{\mathcal{A}}\nabla_\ast u \right|^2
 + \left|\mathbb{D}_{\mathcal{A}}\nabla_\ast^2 u \right|^2 \right)J,
\end{align*}
we can prove the following result.
\begin{theorem}\label{thm:horizontal}
 Suppose $\left(u(t), p(t), \eta(t) \right)$ are solutions of \eqref{eq:new_ns} for $t\in [0,T]$ and $\sup\limits_{t\in[0,T]}\mathcal{E}(t)\le\delta$
for the universal constant $\delta\in(0,1)$ given in Theorem \ref{thm:elliptic}. Then, for all $t\in [0,T]$,
   \[
   \begin{aligned}
     \f{d}{dt} \left(\mathcal{E}_\shortparallel-\mathcal{P} \right)+\mathcal{D}_\shortparallel\lesssim \mathcal{E}^{1/2}\mathcal{D},
   \end{aligned}
   \]
   where
   \[
   \mathcal{P}=\int_\Om pF^2J.
   \]
 \end{theorem}
 \begin{proof}
    Based on Lemma \ref{lem:energy}, we are supposed to take $(v, q, \xi)=(u, p,\eta)$, $(\pa_tu, \pa_tp, \pa_t\eta)$, $(\nabla_\ast u, \nabla_\ast p, \nabla_\ast\eta)$, and $(\nabla_\ast^2u, \nabla_\ast^2p, \nabla_\ast^2\eta)$, respectively, and estimate the corresponding forcing terms in the right-hand side of \eqref{eq:energy_1}. The details of the proof is completely  contained in Propositions \ref{prop:horizon_1}--\ref{prop:horizon_3}.
 \end{proof}

\subsubsection{The estimates for one temporal derivative}

  When $(v,q,\xi)=(\pa_tu,\pa_tp,\pa_t\eta)$,
  \begin{align*}
F^1&=-\dive_{\pa_t\mathcal{A}}S_{\mathcal{A}}(u,p)
  +\dive_{\mathcal{A}}\mathbb{D}_{\pa_t\mathcal{A}}u-\pa_tu\cdot\nabla_{\mathcal{A}}u -u\cdot\nabla_{\pa_t\mathcal{A}}u
  +\pa_t^2\bar{\eta}KW\pa_3u+\pa_t\bar{\eta}\pa_tKW\pa_3u,\\
F^2&=-\dive_{\pa_t\mathcal{A}}u, \\
F^4 &  =u\cdot\pa_t\mathcal{N},\\
F^3& = \eta\pa_t\mathcal{N}-\Delta_\ast\eta\pa_t\mathcal{N}+\mathbb{D}_{\pa_t\mathcal{A}}u\mathcal{N}- S_{\mathcal{A}}(p,u)\pa_t\mathcal{N}
-\frac{ 3 \sum\limits_{i,j=1}^2 \pa_i\eta\pa_j\eta\pa_{ij}\eta
\left(\nabla_\ast\eta\cdot\nabla_\ast\pa_t\eta \right)}{ \left(1+ \left|\nabla_\ast\eta \right|^2 \right)^{5/2}}\mathcal{N}  \\
&\quad+\frac{ \left|\nabla_\ast\eta \right|^2\Delta_\ast\eta}{\sqrt{1+ \left|\nabla_\ast\eta \right|^2} \left(1+\sqrt{1+ \left|\nabla_\ast\eta \right|^2} \right)}\pa_t\mathcal{N}
  +\frac{\sum\limits_{i,j=1}^2\pa_i\eta\pa_j\eta\pa_{ij}\eta}{ \left(1+ \left|\nabla_\ast\eta \right|^2 \right)^{3/2}}\pa_t\mathcal{N}
 \\
&\quad
+\frac{ \left|\nabla_\ast\eta \right|^2
  \Delta_\ast\eta \left(\nabla_\ast\eta\cdot\nabla_\ast\pa_t\eta \right)}{ \left(1+ \left|\nabla_\ast\eta \right|^2 \right)^{3/2} \left(1+\sqrt{1+ \left|\nabla_\ast\eta \right|^2} \right)^2}\mathcal{N}
+\frac{ \left|\nabla_\ast\eta \right|^2\Delta_\ast\pa_t\eta
  + \left(\nabla_\ast\eta\cdot\nabla_\ast\pa_t\eta \right)
  \Delta_\ast\eta}{\sqrt{1+ \left|\nabla_\ast\eta \right|^2} \left(1+\sqrt{1+ \left|\nabla_\ast\eta \right|^2} \right)}\mathcal{N} \\
& \quad   +\frac{\sum\limits_{i,j=1}^2\pa_i\eta\pa_j\eta\pa_{ij}\pa_t\eta}{ \left(1+ \left|\nabla_\ast\eta \right|^2 \right)^{3/2}}\mathcal{N} +\frac{2\sum\limits_{i,j=1}^2\pa_i\pa_t\eta\pa_j\eta\pa_{ij}\eta}{ \left(1+ \left|\nabla_\ast\eta \right|^2 \right)^{3/2}}\mathcal{N}
.
  \end{align*}
We now give the estimates of the interaction between forcing terms and test functions in \eqref{eq:energy_1}.
\begin{proposition}\label{prop:horizon_1}
 \[
 \left\|F^1J \right\|_{H^0}^2+ \left\|F^3 \right\|_{H^0}^2+ \left\|F^4 \right\|_{H^0}^2\lesssim  \left(\mathcal{E}+\mathcal{E}^2 \right)\mathcal{D}.
 \]
\end{proposition}
 \begin{proof}
By the H\"older inequalities, the Sobolev inequality and trace theory along with \eqref{est:bound_coeffi}, we have 
\[
\begin{aligned}
\left\|F^1J \right\|_{H^0}^2&\lesssim\int_{\Om} \left|\dive_{\pa_t\mathcal{A}}S_{\mathcal{A}}(u,p) \right|^2+ \left|\dive_{\mathcal{A}}\mathbb{D}_{\pa_t\mathcal{A}}u \right|^2
+ \left|\pa_tu\cdot\nabla_{\mathcal{A}}u \right|^2+ \left|u\cdot\nabla_{\pa_t\mathcal{A}}u \right|^2\\
  &\quad+ \left|\pa_t^2\bar{\eta}KW\pa_3u \right|^2+ \left|\pa_t\bar{\eta}\pa_tKW\pa_3u \right|^2\\
  &\lesssim \left\|\nabla\pa_t\bar{\eta} \right\|_{L^4}^2 \left( \|\nabla p\|_{L^4}^2
  + \left\|\nabla^2u \right\|_{L^4}^2 \right)
+ \left\|\nabla\pa_t\bar{\eta} \right\|_{L^6}^2 \left\|\nabla^2\bar{\eta} \right\|_{L^6}^2 \|\nabla u\|_{L^6}^2
+ \left\|\nabla^2\pa_t\bar{\eta} \right\|_{L^4}^2  \|\nabla u\|_{L^4}^2 \\
  &\quad+ \left\|\pa_tu \right\|_{L^4}^2 \|\nabla u\|_{L^4}^2+
  \|u\|_{L^6}^2\|\nabla u\|_{L^6}^2 \left\|\nabla\pa_t\bar{\eta} \right\|_{L^6}^2
+ \left\|\pa_t^2\bar{\eta} \right\|_{L^4}^2 \|\nabla u\|_{L^4}^2
+ \left\|\pa_t\bar{\eta} \right\|_{L^6}^2 \left\|\nabla\pa_t\bar{\eta} \right\|_{L^6}^2 \|\nabla u\|_{L^6}^2\\
  &\lesssim  \left\|\pa_t\eta \right\|_{3/2}^2 \left(\|p\|_2^2+\|u\|_3^2 \right)+ \left( \left\|\pa_t\eta \right\|_{5/2}^2
+ \|\pa_tu\|_1^2+ \left\|\pa_t^2\eta \right\|_{1/2}^2 \right)\|u\|_2^2
+\|u\|_2^4 \left\|\pa_t\eta \right\|_{3/2}^2+\|u\|_2^2 \left\|\pa_t\eta \right\|_{3/2}^4\\
& \lesssim \left(\mathcal{E}+\mathcal{E}^2 \right)\mathcal{D}.
\end{aligned}
   \]
By \eqref{est:bound_coeffi} and the Sobolev inequalities, we have 
   \[
   \begin{aligned}
\left\|F^3 \right\|_0^2&\lesssim
 \left\| \left(\eta-\Delta_\ast\eta+p \right)\nabla\pa_t\eta \right\|_0^2
 + \left\|\nabla\pa_t\bar{\eta}\nabla u \right\|_{H^0(\Sigma)}^2
 + \left\|\left|\nabla_\ast\eta \right|^2\nabla_\ast\pa_t\eta \right\|_0^2\\
   &\lesssim \left\|\eta \right\|_{7/2}^2 \left\|\pa_t\eta \right\|_1^2
+ \left(\|p\|_1^2+\|\eta\|_3^2 \right) \left\|\pa_t\eta \right\|_{5/2}^2+\|u\|_0^2 \|\pa_t\eta\|_{3/2}^2
\lesssim \mathcal{E}\mathcal{D}.
   \end{aligned}
   \]
Finally, by the Sobolev inequalities, we have 
\[
    \left\|F^4\right\|_0^2\lesssim \left\|u\nabla\pa_t\eta \right\|_{H^0(\Sigma)}^2\lesssim \|u\|_2^2 \left\|\pa_t\eta \right\|_1^2\lesssim\mathcal{E}\mathcal{D}.   \]
\end{proof}
We now consider the estimate of the pressure term. Since we do not have the estimate for $\pa_tp$, so we turn to the following argument.
 \begin{proposition}\label{prop:horizon_2}
   \[
   \int_\Om \pa_tp F^2J- \f{d}{dt}\int_\Om pF^2J\lesssim \left(\mathcal{E}^{1/2}+\mathcal{E} \right)\mathcal{D}.
   \]
 \end{proposition}
 \begin{proof}
 We rewrite the integration as
   \[
   \int_\Om \pa_tp F^2 J=\f{d}{dt}\int_\Om pF^2J-\int_\Om p\pa_tF^2J-\int_\Om pF^2\pa_tJ.
   \]
Direct calculation yields 
   \[
   \begin{aligned}
    & \quad   \left|\int_\Om p\pa_tF^2J +\int_\Om pF^2\pa_tJ\right| \\
&\lesssim\int_\Om|p| \left( \left|\nabla\pa_t^2\bar{\eta} \right| \left|\nabla u \right|
+ \left|\nabla\pa_t\bar{\eta} \right| \left|\nabla\pa_tu \right|+ \left|\nabla\pa_t\bar{\eta} \right|^2 \left|\nabla u \right| \right)\\
&\lesssim \|p\|_{L^4} \left( \left\|\nabla\pa_t^2\bar{\eta} \right\|_{L^2} \|\nabla u\|_{L^4}
+ \left\|\nabla\pa_t\bar{\eta} \right\|_{L^4} \|\nabla\pa_tu\|_{L^2} \right)
+\|p\|_{L^4} \left\|\nabla\pa_t\bar{\eta} \right\|_{L^4}^2\|\nabla u\|_{L^4}\\
     &\lesssim \|p\|_1^2 \left( \left\|\pa_t^2\eta \right\|_{1/2}\|u\|_2
+ \left\|\pa_t\eta \right\|_{3/2} \|\pa_tu\|_1+ \|\pa_t\eta\|_{3/2}^2\|u\|_2 \right)
\lesssim  \left(\mathcal{E}^{1/2}+\mathcal{E} \right)\mathcal{D}.
   \end{aligned}
   \]
 \end{proof}
Hence, we have
 \[
 \begin{aligned}
 \f{d}{dt}\int_\Om\left( \left|\pa_tu \right|^2J+p \left(\dive_{\pa_t\mathcal{A}}u \right)J\right)
+\f{d}{dt} \left\|\pa_t\eta \right\|_{H^1(\Sigma)}^2\int_\Om \left|\mathbb{D}_{\mathcal{A}}\pa_tu \right|^2J
\lesssim  \left(\mathcal{E}^{1/2}+\mathcal{E} \right)\mathcal{D}.
\end{aligned}
  \]
 Similarly, we also have
 \[
 \begin{aligned}
 \f{d}{dt}\int_\Om|u|^2J+\f{d}{dt} \|\eta\|_{H^1(\Sigma)}^2+\int_\Om \left|\mathbb{D}_{\mathcal{A}}u \right|^2J\lesssim \left(\mathcal{E}^{1/2}+\mathcal{E} \right)\mathcal{D}.
 \end{aligned}
  \]

\subsubsection{The estimates for spatial derivatives}
We now give the explicit forms of forcing terms in \eqref{eq:energy_1} when we take spatial derivatives on $ \left(u,p,\eta \right)$. We will use Einstein summation for the sake of simplicity. When $(v,q,\zeta)= \left(\pa_ju,\pa_jp,\pa_j\eta \right)$, $ j=1, 2$, we have
\[
\begin{aligned}
    F^1&=\pa_j\pa_t\bar{\eta}KW\pa_3u+\pa_t\bar{\eta}\pa_jKW\pa_3u-u\cdot\nabla_{\pa_j\mathcal{A}}u-\pa_j u\cdot\nabla_{\mathcal{A}}u-\nabla_{\pa_j\mathcal{A}}p
+\dive_{\pa_j\mathcal{A}}\mathbb{D}_{\mathcal{A}}u+\dive_{\mathcal{A}}\mathbb{D}_{\pa_j\mathcal{A}}u,\\
  F^2& =-\dive_{\pa_j\mathcal{A}}u, \\
 F^4 &  =-u\cdot\pa_j\mathcal{N},\\
F^3&= \left(\eta-\Delta_\ast\eta \right)\pa_j\mathcal{N}+\mathbb{D}_{\pa_j\mathcal{A}}u\mathcal{N}-S_{\mathcal{A}}(p,u)\pa_j\mathcal{N}
+\frac{ \left|\nabla_\ast\eta \right|^2\Delta_\ast\eta}{\sqrt{1+ \left|\nabla_\ast\eta \right|^2} \left(1+\sqrt{1+ \left|\nabla_\ast\eta \right|^2} \right)}\pa_j\mathcal{N}
\\
  &\quad
  +\frac{\sum_{k,\ell=1}^2\pa_k\eta\pa_\ell\eta\pa_{k\ell}\eta}{ \left(1+ \left|\nabla_\ast\eta \right|^2 \right)^{3/2}}\pa_j\mathcal{N}+\frac{ \left|\nabla_\ast\eta \right|^2\Delta_\ast\pa_j\eta
+ \left(\nabla_\ast\eta\cdot\nabla_\ast\pa_j\eta \right)\Delta_\ast\eta}{\sqrt{1+ \left|\nabla_\ast\eta \right|^2}
 \left(1+\sqrt{1+ \left|\nabla_\ast\eta \right|^2}\right)}\mathcal{N}
 +\frac{2\sum_{k,\ell=1}^2\pa_k\pa_j\eta\pa_\ell\eta\pa_{k\ell}\eta}{\left(1+ \left|\nabla_\ast\eta \right|^2 \right)^{3/2}}\mathcal{N}\\
  &\quad+\frac{ \left|\nabla_\ast\eta \right|^2
\Delta_\ast\eta(\nabla_\ast\eta\cdot\nabla_\ast\pa_j\eta)}{ \left(1+ \left|\nabla_\ast\eta \right|^2 \right)^{3/2} \left(1+\sqrt{1+ \left|\nabla_\ast\eta \right|^2} \right)}
  \mathcal{N}+\frac{\sum_{k,\ell=1}^2\pa_k\eta\pa_\ell\eta\pa_{k\ell}\pa_j\eta}{ \left(1+ \left|\nabla_\ast\eta \right|^2 \right)^{3/2}}\mathcal{N}\\
&+\frac{ \left|\nabla_\ast\eta \right|^2\Delta_\ast\eta \left(\nabla_\ast\eta\cdot\nabla_\ast\pa_j\eta \right)}{ \left(1+ \left|\nabla_\ast\eta \right|^2 \right)\left(1+\sqrt{1+ \left|\nabla_\ast\eta \right|^2} \right)^2}\mathcal{N}
-3\frac{\sum_{k,\ell=1}^2\pa_k\eta\pa_\ell\eta\pa_{k\ell}\eta \left(\nabla_\ast\eta\cdot\nabla_\ast\pa_j\eta \right)}
{\left(1+ \left|\nabla_\ast\eta \right|^2 \right)^{5/2}}\mathcal{N}. 
\end{aligned}
 \]
 We also need to take horizontal derivatives on $(u,p,\eta)$ twice in our definitions of horizontal energy and dissipation.
Taking $(v, q, \zeta)=(\pa_{ij}u, \pa_{ij}p,\pa_{ij}\eta)$, $i,j=1,2$, we have 
  \[
  \begin{aligned}
F^1&=\pa_{ij}\pa_t\bar{\eta}KW\pa_3u+\pa_j\pa_t\bar{\eta}\pa_iKW\pa_3u+\pa_j\pa_t\bar{\eta}KW\pa_{i3}u+\pa_i\pa_t\bar{\eta}\pa_jKW\pa_3u
+\pa_t\bar{\eta}\pa_{ij}KW\pa_3u
+\pa_t\bar{\eta}\pa_jKW\pa_{3i}u\\
&\quad
-\pa_iu\cdot\nabla_{\pa_j\mathcal{A}}u-u\cdot\nabla_{\pa_{ij}\mathcal{A}}u-u\cdot\nabla_{\pa_j\mathcal{A}}\pa_iu-\pa_{ij} u\cdot\nabla_{\mathcal{A}}u
-\pa_ju\cdot\nabla_{\pa_i\mathcal{A}}u
-\pa_ju\cdot\nabla_{\mathcal{A}}\pa_i u +\dive_{\mathcal{A}}\mathbb{D}_{\pa_{ij}\mathcal{A}}u
\\
    &\quad
-\nabla_{\pa_j\mathcal{A}}\pa_ip+\dive_{\pa_{ij}\mathcal{A}}\mathbb{D}_{\mathcal{A}}u
    +\dive_{\pa_j\mathcal{A}}\mathbb{D}_{\pa_i\mathcal{A}}u
    +\dive_{\pa_j\mathcal{A}}\mathbb{D}_{\mathcal{A}}\pa_i u
 +\dive_{\pa_i\mathcal{A}}\mathbb{D}_{\pa_j\mathcal{A}}u +\dive_{\mathcal{A}}\mathbb{D}_{\pa_j\mathcal{A}}\pa_i u
 -\nabla_{\pa_{ij}\mathcal{A}}p
 \\
&\quad +\pa_t\bar{\eta}\pa_iKW\pa_{3j}u+\pa_i\pa_t\bar{\eta}KW\pa_{3j}u
-\pa_iu\cdot\nabla_{\mathcal{A}}\pa_j u-u\cdot\nabla_{\pa_i\mathcal{A}}\pa_j u-\dive_{\pa_i\mathcal{A}}S_{\mathcal{A}}(\pa_jp,\pa_ju)+\dive_{\mathcal{A}}\mathbb{D}_{\pa_i\mathcal{A}}\pa_ju,\\
 F^2&=-\dive_{\pa_{ij}\mathcal{A}}u-\dive_{\pa_j\mathcal{A}}\pa_iu-\dive_{\pa_i\mathcal{A}}\pa_ju,
\\
  F^4 &   = -\pa_iu\cdot\pa_j\mathcal{N}-u\cdot\pa_{ij}\mathcal{N}-\pa_ju\cdot\pa_i\mathcal{N},\\
\end{aligned}
 \]
 \[
 \begin{aligned}
    F^3& = \left(\pa_i\eta-\Delta_\ast\pa_i\eta \right)\pa_j\mathcal{N}
+\left(\eta-\Delta_\ast\eta \right)\pa_{ij}\mathcal{N}
+\mathbb{D}_{\pa_{ij}\mathcal{A}}u\mathcal{N}
+\mathbb{D}_{\pa_j\mathcal{A}}\pa_iu\mathcal{N}
+\mathbb{D}_{\pa_j\mathcal{A}}u\pa_i\mathcal{N}-S_{\mathcal{A}} \left(\pa_ip,\pa_iu \right)\pa_j\mathcal{N}
    \\
&\quad
+2\mathbb{D}_{\pa_i\mathcal{A}}\pa_ju\mathcal{N}-S_{\mathcal{A}}\left(\pa_jp,\pa_ju \right)\pa_i\mathcal{N}
-S_{\mathcal{A}}(p,u)\pa_{ji}\mathcal{N}
+\frac{2 \left|\nabla_\ast\eta \right|^2\Delta_\ast\pa_i\eta}{\sqrt{1+ \left|\nabla_\ast\eta \right|^2}
  \left(1+\sqrt{1+ \left|\nabla_\ast\eta \right|^2} \right)}\pa_j\mathcal{N}
  \\
  &\quad
  +\frac{3 \left(\nabla_\ast\eta\cdot\nabla_\ast\pa_i\eta \right)\Delta_\ast\eta}{\sqrt{1+ \left|\nabla_\ast\eta \right|^2} \left(1+\sqrt{1+ \left|\nabla_\ast\eta \right|^2} \right)}\pa_j\mathcal{N}
  +\frac{ \left|\nabla_\ast\eta \right|^2\Delta_\ast\pa_i\eta}{\sqrt{1+ \left|\nabla_\ast\eta \right|^2}
  \left(1+\sqrt{1+ \left|\nabla_\ast\eta \right|^2} \right)}\pa_{ij}\mathcal{N}
  \\
  &\quad
+\frac{\sum_{k,\ell=1}^2\pa_k\eta\pa_\ell\eta\pa_{k\ell}\eta}{ \left(1+ \left|\nabla_\ast\eta \right|^2 \right)^{3/2}}\pa_{ij}\mathcal{N}
 +  \frac{ 2 \pa_k\eta\pa_\ell\eta\pa_{ik\ell}\eta}{ \left(1+ \left|\nabla_\ast\eta \right|^2 \right)^{3/2}}\pa_j\mathcal{N}
-\frac{ 6 \pa_k\eta\pa_\ell\eta\pa_{k\ell}\eta \left(\nabla_\ast\eta\cdot\nabla_\ast\pa_i\eta \right)}
{ \left(1+|\nabla_\ast\eta|^2 \right)^{5/2}}\pa_j\mathcal{N}
\\
  &\quad
+3\frac{\pa_i \left(\pa_k\eta\pa_\ell\eta \right)\pa_{k\ell}\eta}{ \left(1+ \left|\nabla_\ast\eta \right|^2 \right)^{3/2}}\pa_j\mathcal{N}
-\frac{ 3 \pa_k\eta\pa_\ell\eta\pa_{jk\ell}\eta
  \left(\nabla_\ast\eta\cdot\nabla_\ast\pa_i\eta \right)}{ \left(1+ \left|\nabla_\ast\eta \right|^2 \right)^{5/2}}\mathcal{N}
+\frac{ \left|\nabla_\ast\eta \right|^2\Delta_\ast\pa_{ij}\eta
+ \left(\nabla_\ast\eta\cdot\nabla_\ast\pa_{ij}\eta \right)\Delta_\ast\eta}{\sqrt{1+ \left|\nabla_\ast\eta \right|^2}
\left(1+\sqrt{1+ \left|\nabla_\ast\eta \right|^2} \right)}\mathcal{N}
\\
  &\quad
  +\frac{3 \left(\nabla_\ast\eta\cdot\nabla_\ast\pa_i\eta \right)\Delta_\ast\pa_j\eta
+ \left(\nabla_\ast\pa_i\eta\cdot\nabla_\ast\pa_j\eta \right)\Delta_\ast\eta}{\sqrt{1+ \left|\nabla_\ast\eta \right|^2}
\left(1+\sqrt{1+ \left|\nabla_\ast\eta \right|^2} \right)}\mathcal{N}
-\frac{ \left|\nabla_\ast\eta \right|^2\Delta_\ast\pa_j\eta
   \left(\nabla_\ast\eta\cdot\nabla_\ast\pa_i\eta \right)}{ \left(1+ \left|\nabla_\ast\eta \right|^2 \right)^{3/2}
  \left(1+\sqrt{1+ \left|\nabla_\ast\eta \right|^2} \right)}\mathcal{N}
\\
  &\quad
  -\frac{ \left|\nabla_\ast\eta \right|^2\Delta_\ast\pa_j\eta \left(\nabla_\ast\eta\cdot\nabla_\ast\pa_i\eta \right)}
  { \left(1+ \left|\nabla_\ast\eta \right|^2 \right)\left(1+\sqrt{1+ \left|\nabla_\ast\eta \right|^2} \right)^2}\mathcal{N}+\frac{ 15 \left(\pa_k\eta\pa_\ell\eta\pa_{k\ell}\eta \right) \left(\nabla_\ast\eta\cdot\nabla_\ast\pa_j\eta \right) \left(\nabla_\ast\eta\cdot\nabla_\ast\pa_i\eta \right)}
  { \left(1+ \left|\nabla_\ast\eta \right|^2 \right)^{7/2}}\mathcal{N}
\\
&\quad
+\frac{ \left|\nabla_\ast\eta \right|^2\Delta_\ast\pa_i\eta \left(\nabla_\ast\eta\cdot\nabla_\ast\pa_j\eta \right)}
  { \left(1+ \left|\nabla_\ast\eta \right|^2 \right)^{3/2} \left(1+\sqrt{1+ \left|\nabla_\ast\eta \right|^2} \right)}\mathcal{N}
  +\frac{ \left|\nabla_\ast\eta \right|^2\Delta_\ast\pa_i\eta \left(\nabla_\ast\eta\cdot\nabla_\ast\pa_j\eta \right)}
  { \left(1+|\nabla_\ast\eta|^2 \right) \left(1+\sqrt{1+|\nabla_\ast\eta|^2} \right)^2}\mathcal{N}
\\
  &\quad
+\frac{ \left|\nabla_\ast\eta \right|^2\Delta_\ast\eta \left(\nabla_\ast\pa_i\eta\cdot\nabla_\ast\pa_j\eta \right)}
  { \left(1+ \left|\nabla_\ast\eta \right|^2 \right)^{3/2} \left(1+\sqrt{1+ \left|\nabla_\ast\eta \right|^2} \right)}\mathcal{N}
  +\frac{ \left|\nabla_\ast\eta \right|^2\Delta_\ast\eta \left(\nabla_\ast\eta\cdot\nabla_\ast\pa_{ij}\eta \right)}
  {\left(1+ \left|\nabla_\ast\eta \right|^2 \right)^{3/2}
  \left(1+\sqrt{1+ \left|\nabla_\ast\eta \right|^2} \right)}\mathcal{N}
\\
  &\quad
-\frac{ 3  \left|\nabla_\ast\eta \right|^2\Delta_\ast\eta \left(\nabla_\ast\eta\cdot\nabla_\ast\pa_j\eta \right) \left(\nabla_\ast\eta\cdot\nabla_\ast\pa_i\eta \right)}
  { \left(1+ \left|\nabla_\ast\eta \right|^2 \right)^{5/2} \left( 1+\sqrt{1+ \left|\nabla_\ast\eta \right|^2} \right)}\mathcal{N}
  +\frac{ \left|\nabla_\ast\eta \right|^2\Delta_\ast\eta \left(\nabla_\ast\pa_i\eta\cdot\nabla_\ast\pa_j\eta \right)}{ \left(1+ \left|\nabla_\ast\eta \right|^2 \right)
  \left(1+\sqrt{1+ \left|\nabla_\ast\eta \right|^2} \right)^2}\mathcal{N}
\\
  &\quad
+\frac{ \left|\nabla_\ast\eta \right|^2\Delta_\ast\eta \left(\nabla_\ast\eta\cdot\nabla_\ast\pa_{ij}\eta \right)}
  { \left(1+ \left|\nabla_\ast\eta \right|^2 \right) \left(1+\sqrt{1+ \left|\nabla_\ast\eta \right|^2} \right)^2}\mathcal{N}
  - \frac{ 2 \left|\nabla_\ast\eta \right|^2\Delta_\ast\eta \left(\nabla_\ast\eta\cdot\nabla_\ast\pa_j\eta \right) \left(\nabla_\ast\eta\cdot\nabla_\ast\pa_i\eta \right)}
  { \left(1+ \left|\nabla_\ast\eta \right|^2 \right)^{3/2} \left(1+\sqrt{1+ \left|\nabla_\ast\eta \right|^2} \right)^4}\mathcal{N}
\\
&\quad
+\frac{\pa_{ki}\eta\pa_\ell\eta\pa_{jk\ell}\eta+\pa_k\eta\pa_{i\ell}\eta\pa_{jk\ell}\eta+\pa_k\eta\pa_{\ell}\eta\pa_{ijk\ell}\eta}
  { \left(1+ \left|\nabla_\ast\eta \right|^2 \right)^{3/2}}\mathcal{N} +\frac{2\pa_{kj}\eta\pa_{\ell}\eta\pa_{ik\ell}\eta}{\left(1+ \left|\nabla_\ast\eta \right|^2 \right)^{3/2}}\mathcal{N} +\frac{2\pa_{kj}\eta\pa_{i\ell}\eta\pa_{k\ell}\eta}{ \left(1+ \left|\nabla_\ast\eta \right|^2 \right)^{3/2}}\mathcal{N}
\\
&\quad
- \frac{ 3 \pa_i \left(\pa_k\eta\pa_\ell\eta\pa_{k\ell}\eta \right)
  \left(\nabla_\ast\eta\cdot\nabla_\ast\pa_j\eta \right)+ \left(\pa_k\eta\pa_\ell\eta\pa_{k\ell}\eta \right)\pa_i \left(\nabla_\ast\eta\cdot\nabla_\ast\pa_j\eta \right)}
  { \left(1+ \left|\nabla_\ast\eta \right|^2 \right)^{5/2}}\mathcal{N}
+\frac{2\pa_{ikj}\eta\pa_\ell\eta\pa_{k\ell}\eta}{ \left(1+ \left|\nabla_\ast\eta \right|^2 \right)^{3/2}}\mathcal{N}
  .
  \end{aligned}  \]
We can now obtain the following result.
\begin{proposition}\label{prop:horizon_3}
    \[
    \left\|F^1\right\|_0+ \left\|F^2 \right\|_0+ \left\|F^3 \right\|_{-1/2}+ \left\|F^4 \right\|_{1/2}\lesssim \left(\mathcal{E}^{1/2}+\mathcal{E} \right)\mathcal{D}^{1/2}.
    \]
  \end{proposition}
  \begin{proof}
  To prove the estimates of forcing terms, it suffices to consider the case of twice spatial derivatives, since the once differentiable case is similar yet much easier. It is easy to see that
  \[
  \begin{aligned}
    \left\|F^1 \right\|_{H^0(\Om)}^2&\lesssim\int_\Om \left|\nabla^2\pa_t\bar{\eta} \right|^2|\nabla u|^2+ \left|\nabla\pa_t\bar{\eta} \right|^2 \left|\nabla^2\bar{\eta} \right|^2 |\nabla u|^2+ \left|\nabla\pa_t\bar{\eta} \right|^2
  \left|\nabla^2u \right|^2+ \left|\pa_t\bar{\eta} \right|^2 \left|\nabla^3\bar{\eta} \right|^2|\nabla u|^2\\
    &\quad+\int_\Om \left|\pa_t\bar{\eta} \right|^2 \left|\nabla^2\bar{\eta} \right|^2 \left|\nabla^2 u \right|^2+ \left|\nabla^2\bar{\eta} \right|^2
    |\nabla u|^4+ |u|^2 \left|\nabla^3\bar{\eta} \right|^2 |\nabla u|^2 +|u|^2 \left|\nabla^2\bar{\eta} \right|^2 \left|\nabla^2 u\right|^2\\
    &\quad+\int_\Om \left|\nabla u \right|^2  \left|\nabla^2u \right|^2+ \left|\nabla^3\bar{\eta} \right|^2 \left|\nabla p \right|^2+ \left|\nabla^2\bar{\eta} \right|^2  \left|\nabla^2 p \right|^2+ \left|\nabla^3\bar{\eta} \right|^2 \left( \left|\nabla^2\bar{\eta} \right|^2 |\nabla u|^2+ \left|\nabla^2u \right|^2 \right)\\
    &\quad+\int_\Om \left|\nabla^2\bar{\eta} \right|^4 \left|\nabla^2u \right|^2+ \left|\nabla^2\bar{\eta} \right|^2
    \left|\nabla^2u \right|^2+ \left|\nabla^4\bar{\eta} \right|^2  |\nabla u|^2.
  \end{aligned}
  \]
  Then by the H\"older inequality, the Sobolev inequality and usual trace theory, we have 
  \[
  \begin{aligned}
\left\|F^1 \right\|_{H^0}^2 &\lesssim \left( \left\|\nabla^2\pa_t\bar{\eta} \right\|_{L^2}^2+ \left\|\nabla\pa_t\bar{\eta} \right\|_{L^4}^2 \left\|\nabla^2\bar{\eta} \right\|_{L^4}^2 \right) \|\nabla u\|_{L^\infty}^2+ \left\|\nabla\pa_t\bar{\eta} \right\|_{L^4}^2 \left\|\nabla^2u \right\|_{L^4}^2\\
    &\quad+ \left\|\pa_t\bar{\eta} \right\|_{L^\infty}^2 \left\|\nabla^3\bar{\eta} \right\|_{L^2}^2
    \|\nabla u\|_{L^\infty}^2+ \left\|\pa_t\bar{\eta} \right\|_{L^6}^2 \left\|\nabla^2\bar{\eta} \right\|_{L^6}^2
    \|\nabla u\|_{L^6}^2+ \left\|\nabla^4\bar{\eta} \right\|_{L^2}^2 \|\nabla u\|_{L^\infty}^2\\
    &\quad+ \left\|\nabla^2\bar{\eta} \right\|_{L^4}^2
    \|\nabla u\|_{L^4}^2
    \|\nabla u\|_{L^\infty}^2+\|u\|_{L^\infty}^2 \left\|\nabla^3\bar{\eta} \right\|_{L^2}^2
    \|\nabla u\|_{L^\infty}^2+\|u\|_{L^\infty}^2 \left\|\nabla^2\bar{\eta} \right\|_{L^\infty}^2 \left\|\nabla^2 u \right\|_{L^2}^2\\
    &\lesssim \left(\mathcal{E}+\mathcal{E}^2 \right)\mathcal{D}.
  \end{aligned}
  \]
Similarly, for $F^2$, we have 
\[
  \begin{aligned}
\left\|F^2 \right\|_0^2&\lesssim\int_\Om  \left|\nabla_\ast^2\nabla\bar{\eta} \right|^2 \left|\nabla u \right|^2+ \left|\nabla_\ast\nabla\bar{\eta} \right|^2|\nabla_\ast\nabla u|^2\\
&\lesssim \left\| \left|\nabla_\ast^2\nabla\bar{\eta} \right|^2 \right\|_{L^3}^2
\left\|\nabla u \right\|_{L^6}^2 + \left\| \left|\nabla_\ast\nabla\bar{\eta} \right|^2 \right\|_{L^\infty}^2
\left\|\nabla\nabla_\ast u \right\|_{L^2}^2 \lesssim\|\eta\|_3^2\|u\|_2^2+\|\eta\|_{7/2}^2\|u\|_2\lesssim\mathcal{E}\mathcal{D}.
\end{aligned}
  \]
We now turn to the estimate of term $F^3$. We first focus on the terms of $\eta$ with derivatives of $4$-th order. For $v\in H^1(\Om)$ with trace $v|_\Sigma\in H^{1/2}$, we will use the dual form and product of estimates in Sobolev spaces, and see
\[
\begin{aligned}
&\left|\int_\Sigma\left(\f{\pa_k\eta\pa_{\ell}\eta\pa_{ijk\ell}\eta}
{(1+ \left|\nabla_\ast\eta \right|^2)^{3/2}} +\frac{ \left|\nabla_\ast\eta \right|^2
\Delta_\ast\pa_{ij}\eta}{\sqrt{1+ \left|\nabla_\ast\eta \right|^2}(1+\sqrt{1+ \left|\nabla_\ast\eta \right|^2})}\right)\mathcal{N}\cdot v\right|\\
&\lesssim\left|\int_\Sigma\left(\pa_k\eta\pa_{\ell}\eta\pa_{ijk\ell}\eta + \left|\nabla_\ast\eta \right|^2\Delta_\ast\pa_{ij}\eta\right)\mathcal{N}\cdot v\right|
\lesssim  \left\|\nabla_\ast^4\eta \right\|_{-1/2,\Sigma} \left\|\nabla_\ast\eta\nabla_\ast\eta \left(v\cdot\mathcal{N} \right) \right\|_{1/2,\Sigma}
\lesssim \|\eta\|_{7/2}  \|\eta\|_3^2\|v\|_{1/2,\Sigma}.
\end{aligned}
\]
The other terms in $F^3$ can be bounded by
\[
\begin{aligned}
& \int_\Sigma \left|\nabla\nabla_\ast u \right|^2
\left|\nabla_\ast\nabla\bar{\eta} \right|^2+ \left|\nabla u \right|^2 \left|\nabla_\ast^2\nabla\bar{\eta} \right|^2
+ \left( |\nabla_\ast p|^2+ \left|\nabla_\ast\nabla\bar{\eta} \right|^2 |\nabla u|^2
+ \left|\nabla \nabla_\ast u \right|^2 \right) \left|\nabla_\ast^2\eta \right|^2 \\
&\quad+
\int_\Sigma\left( |p|^2+ \left|\nabla u \right|^2\right)
\left|\nabla_\ast^3\eta \right|^2
+  \int_\Sigma \left|\nabla_\ast\eta \right|^2 \left|\nabla_\ast^2\eta \right|^2
+ \left( \left|\eta \right|^2+ \left|\nabla_\ast\eta \right|^4+ \left|\nabla_\ast^2\eta \right|^2 \right) \left|\nabla_\ast^3\eta\right|^2 \\
      &\quad+ \int_\Sigma\left( \left|\nabla_\ast\eta \right|^2+
\left|\nabla_\ast\eta \right|^6 \right) \left|\nabla_\ast^2\eta \right|^2
\left|\nabla_\ast^3\eta \right|^2+
\left|\nabla_\ast\eta \right|^4 \left|\nabla_\ast^3\eta \right|^4
+\left(1+ \left|\nabla_\ast\eta \right|^2+ \left|\nabla_\ast\eta \right|^6 \right) \left|\nabla_\ast^2\eta \right|^6 \\
&\quad+  \int_\Sigma \left|\nabla_\ast^2\eta \right|^4
      \left|\nabla_\ast^3\eta \right|^2
+\left( \left|\nabla_\ast\eta \right|^8+ \left|\nabla_\ast\eta \right|^4 \right)
\left| \nabla_\ast^2\eta \right|^6+  \left|\nabla_\ast\eta \right|^4 \left|\nabla_\ast^4\eta \right|^2 \left|\nabla_\ast^3\eta \right|^2.
    \end{aligned}
    \]
It turns out that the above terms can be bounded further by $(\mathcal{E}+\mathcal{E}^2)\mathcal{D}$.

Finally, we apply similar arguments as in the estimate of $F^3$ to estimate $F^4$,
\[
\begin{aligned}
\left\|F^4 \right\|_{1/2}^2&\lesssim  \left\|\nabla u_\ast\cdot\nabla_\ast^2\eta \right\|_{1/2,\Sigma}^2
+ \left\|u_\ast\cdot\nabla_\ast^3\eta \right\|_{1/2,\Sigma}^2\\
&\lesssim \left\|\nabla u \right\|_{H^{3/2}(\Sigma)}^2 \left\|\eta \right\|_3^2
+\|u\|_{L^\infty(\Sigma)}^2 \|\eta\|_{7/2}^2+\|u\|_{W^{1/2, \infty}(\Sigma)} \left\|\nabla^3_\ast\eta \right\|_{L^2}\\
&
\lesssim \|u\|_3^2\|\eta\|_3^2+ \|u\|_2^2  \|\eta\|_{7/2}^2\lesssim\mathcal{E}\mathcal{D}.
\end{aligned}
\]
\end{proof}

\subsection{Elliptic estimates}
In order to close the a priori estimates, we need to employ \eqref{est:elliptic_1} and \eqref{est:elliptic_2}. Since
\[
G^1=-\pa_tu+\pa_t\bar{\eta}KW\pa_3u-u\cdot\nabla_{\mathcal{A}}u+\nabla_{I-\mathcal{A}}p -\dive_{\mathcal{A}}\mathbb{D}_{I-\mathcal{A}}u-\dive_{I-\mathcal{A}}\mathbb{D}u,
\]
\[
G^2=\dive_{I-\mathcal{A}}u,\text{ and } G^3=-\mathbb{D}_{I-\mathcal{A}}u\mathcal{N}\cdot\mathcal{T}-\mathbb{D}u(e_3-\mathcal{N})\cdot\mathcal{T}-\mathbb{D}ue_3\cdot(\tau-\mathcal{T}),
\]
we have the following estimate in order to use the elliptic theory in Theorem \ref{thm:elliptic}.
\begin{proposition}\label{prop:perturb}
\begin{align*}
\left\|G^1 \right\|_0^2+ \left\|G^2 \right\|_1^2+ \left\|G^3 \right\|_{1/2,\Sigma}^2\lesssim\mathcal{E}_\shortparallel+\mathcal{E}^2,
\intertext{ and }
\left\|G^1 \right\|_1^2+ \left\|G^2 \right\|_2^2+ \left\|G^3 \right\|_{3/2,\Sigma}^2\lesssim\mathcal{D}_\shortparallel
+ \left(\mathcal{E}+\mathcal{E}^2 \right)\mathcal{D}.
\end{align*}
\end{proposition}
\begin{proof}
From the expression of $G^1$, we use the product estimates in Sobolev spaces to deduce 
\[
\begin{aligned}
\left\|G^1 \right\|_0^2 &\lesssim  \left\|\pa_tu \right\|_0^2+ \left\|\pa_t\eta \right\|_{3/2}^2 \|u\|_1^2+\|u\|_1^2\|u\|_2^2+
\|\eta\|_3^2\|p\|_1^2  +\|\eta\|_3^2\|u\|_2^2\lesssim \mathcal{E}_\shortparallel+\mathcal{E}^2.
\end{aligned}
\]
Following the same argument, we also have 
  \[
  \begin{aligned}
\left\|G^1 \right\|_1^2&\lesssim  \left\|\pa_tu \right\|_1^2
+ \left\|\pa_t\eta \right\|_{3/2}^2  \|\eta\|_3^2\|u\|_3^2+\|u\|_2^2\|u\|_3^2+ \left\|\nabla\bar{\eta} \right\|_{L^\infty}^2\|p\|_2^2\\
  &\quad+\left\|\nabla\bar{\eta} \right\|_{L^4(\Om)}^2
\|\nabla p\|_{L^4(\Om)}^2 + \left\|\nabla\bar{\eta} \right\|_{L^\infty}^2
\|u\|_3^2+ \left\|\nabla^2\bar{\eta} \right\|_{L^4(\Om)}^2
\left\|\nabla^2 u \right\|_{L^4(\Om)}^2+ \left\|\nabla^3\bar{\eta} \right\|_{L^4(\Om)}^2
\|\nabla u\|_{L^4(\Om)}^2\\
  &\lesssim \|\pa_tu\|_1^2+ \left(1+ \left\|\pa_t\eta \right\|_{3/2}^2 \right)\|\eta\|_3^2
  \|u\|_3^2+\|u\|_2^2\|u\|_3^2+\|\eta\|_3^2\|p\|_2^2
\lesssim \mathcal{D}_\shortparallel+ \left(\mathcal{E}+\mathcal{E}^2 \right)\mathcal{D}.
  \end{aligned}
  \]
Since $G^2=\dive_{I-\mathcal{A}}u$, we directly have
\begin{align*}
\left\|G^2 \right\|_1^2&\lesssim \left\|\nabla\bar{\eta} \right\|_{L^\infty}
\|u\|_2^2+ \left\|\nabla^2\bar{\eta} \right\|_{L^4(\Om)}^2
\|\nabla u\|_{L^4(\Om)}^2
\lesssim \|\eta\|_3^2\|u\|_2^2+\|\eta\|_{5/2}^2 \|u\|_2^2\lesssim \mathcal{E}^2,
\intertext{ and }
\left\|G^2 \right\|_2^2&\lesssim \left\|\nabla\bar{\eta} \right\|_{L^\infty}\|u\|_3^2+
\left\|\nabla^3\bar{\eta} \right\|_2^2 \|\nabla u\|_{L^\infty}^2
\lesssim  \|\eta\|_3^2\|u\|_3^2\lesssim \mathcal{E}\mathcal{D}.
\end{align*}
From the explicit form of $G^3$, we apply the H\"older inequality and Sobolev inequality and get 
  \[
  \begin{aligned}
\left\|G^3 \right\|_{1/2,\Sigma}^2&\lesssim \left(\|\nabla\eta\|_{L^\infty(\Sigma)}^2+ \left\|\nabla\bar{\eta} \right\|_{L^\infty(\Sigma)}^2 \right)
\|\nabla u\|_{1/2,\Sigma}^2
+ \left( \|\nabla\eta\|_{W^{1/2,4}(\Sigma)}^2+ \|\nabla\bar{\eta}\|_{W^{1/2,4}(\Sigma)}^2 \right) \|\nabla u\|_{L^4(\Sigma)}^2\\
    &\lesssim  \|\eta\|_3^2\|u\|_2^2\lesssim \mathcal{E}^2.
  \end{aligned}
  \]
  The same argument enables us to have 
  \[
  \begin{aligned}
\left\|G^3 \right\|_{3/2,\Sigma}^2&\lesssim \left( \left\|\nabla\eta \right\|_{3/2}^2+ \|\nabla\bar{\eta}\|_{3/2,\Sigma}^2 \right)
\|\nabla u\|_{3/2,\Sigma}^2
\lesssim \|\eta\|_3^2\|u\|_3^2\lesssim \mathcal{E}\mathcal{D}.
  \end{aligned}
  \]
Combining the above estimates, we get the results.
\end{proof}
Proposition \ref{prop:perturb} allows us to control the remainder terms in energy $\mathcal{E}$ and dissipation $\mathcal{D}$.
\begin{theorem}\label{thm:enhanced_dissipation}
Under the same assumption in Theorem \ref{thm:horizontal}, we have
\begin{align*}
& \quad \|u\|_2^2+\|p\|_1^2+\|\pa_t\eta\|_{3/2}^2\lesssim \mathcal{E}_\shortparallel+\mathcal{E}^2,
\intertext{  and }
& \|u\|_3^2+\|p\|_2^2+\|\eta\|_{7/2}+ \left\|\pa_t\eta \right\|_{5/2}^2
 + \left\|\pa_t^2\eta \right\|_{1/2}^2\lesssim \mathcal{D}_\shortparallel+\mathcal{E}\mathcal{D}.
\end{align*}
\end{theorem}
\begin{proof}
By $G^4=u\cdot\mathcal{N}$, the Sobolev inequality and Poincar\'e inequality, we get 
\[
\begin{aligned}
\left\|u\cdot\mathcal{N} \right\|_{3/2,\Sigma}^2&\lesssim \|u_3\|_{3/2,\Sigma}^2+ \left\|u_\ast\cdot\nabla_\ast\eta \right\|_{3/2}^2
\lesssim \|u_3\|_{0,\Sigma}^2+ \left\|\nabla_\ast u_3 \right\|_{1/2,\Sigma}^2 +  \|u\|_2^2\|\eta\|_{5/2}^2\\
&\lesssim  \left\|\pa_3u_3 \right\|_0^2 +\|\nabla_\ast u\|_1^2+ \|u\|_2^2  \|\eta\|_{5/2}^2.
\end{aligned}
\]
Then we use $\dive_{\mathcal{A}}u=0$, i.e.,
\[
 K\pa_3 u_3=-\pa_1 u_1- \pa_2 u_2 + AK\pa_3 u_1  + BK\pa_3 u_2.
\]
By the smallness of $\eta$, we can obtain that there exists a constant $\delta_1>0$ such that $1\gtrsim\|J\|_{W^{1,\infty}}\ge\delta_1$,
so that we have 
\[
\left\|K\pa_3 u_3 \right\|_0^2 \gtrsim \left\|\pa_3u_3 \right\|_0^2.
\]
By the H\"older inequality and Sobolev inequality, we have
\[
\left\|AK\pa_3 u_1  + BK\pa_3 u_2 \right\|_0^2\lesssim \|u\|_2^2\|\eta\|_{5/2}^2.
\]
Thus
\[
\left\|\pa_3u_3 \right\|_0^2\lesssim \left\|\nabla_\ast u_\ast \right\|_0^2+\|u\|_2^2\|\eta\|_{5/2}^2.
\]
Therefore, we have
\[
\left\|u\cdot\mathcal{N} \right\|_{3/2,\Sigma}^2\lesssim  \left\|\nabla_\ast u \right\|_1^2++\|u\|_2^2\|\eta\|_{5/2}^2\lesssim \mathcal{E}_\shortparallel+\mathcal{E}^2.
\]
Similarly, we have
  \[
  \begin{aligned}
\left\|u\cdot\mathcal{N} \right\|_{5/2,\Sigma}^2&\lesssim \|u_3\|_{5/2,\Sigma}^2+ \left\|u_\ast \right\|_{5/2,\Sigma}^2
\|\eta\|_{5/2}^2+\|u\|_{3/2,\Sigma}^2\|\eta\|_{7/2}^2\\
  &\lesssim \|u_3\|_{1,\Sigma}^2+ \left\|\nabla_\ast^2 u_3 \right\|_{1/2,\Sigma}^2+  \left\|u \right\|_3^2
  \|\eta\|_{5/2}^2+\|u\|_2^2\|\eta\|_{7/2}^2\\
&\lesssim \left\|\pa_3u_3 \right\|_1^2+  \left\|\nabla_\ast^2 u \right\|_1^2+\|u\|_3^2\|\eta\|_{5/2}^2+\|u\|_2^2\|\eta\|_{7/2}^2.
\end{aligned}
  \]
We still use the divergence condition $\dive_{\mathcal{A}}u=0$ to get
\[
  \|\pa_3u_3\|_1^2\lesssim \|K\pa_3u_3\|_1^2\lesssim \|\nabla_\ast u_\ast\|_1^2+\|u\|_3^2\|\eta\|_{5/2}^2+\|u\|_2^2\|\eta\|_{7/2}^2,
  \]
  which implies
\[
  \|u\cdot\mathcal{N}\|_{5/2,\Sigma}^2\lesssim \mathcal{D}_\shortparallel+\mathcal{E}\mathcal{D}.
  \]
Hence, applying the elliptic theory in Theorem \ref{thm:elliptic} and Proposition \ref{prop:perturb}, we have 
\[
  \|u\|_2^2+\|p\|_1^2\lesssim \mathcal{E}_\shortparallel+\mathcal{E}^2,
\text{   and }
  \|u\|_3^2+\|p\|_2^2+\|\eta\|_{7/2}^2\lesssim \mathcal{D}_\shortparallel+\mathcal{E}\mathcal{D}.
  \]
  Since $\pa_t\eta=u\cdot\mathcal{N}$, by the estimates of $u\cdot\mathcal{N}$, we get
  \[
  \|\pa_t\eta\|_{3/2}^2\lesssim \mathcal{E}_\shortparallel+\mathcal{E}^2,
\text{   and }
  \|\pa_t\eta\|_{5/2}^2\lesssim\mathcal{D}_\shortparallel+\mathcal{E}\mathcal{D}.
  \]
Finally, we take $\pa_t$ on both sides of $\pa_t\eta=u\cdot\mathcal{N}$ and obtain
  \[
  \pa_t^2\eta= \pa_tu\cdot\mathcal{N}+u\cdot\pa_t\mathcal{N}.
  \]
Therefore, we use the horizontal dissipation to get
\[
\begin{aligned}
\left\|\pa_t^2\eta \right\|_{1/2}^2&\lesssim \|\pa_tu\|_{1/2,\Sigma}^2 \left(1+\|\eta\|_3^2 \right)+
\left\|\pa_t\nabla_\ast\eta \right\|_{1/2}^2
\|u\|_{5/2,\Sigma}^2
\lesssim\mathcal{D}_\shortparallel+\mathcal{E}\mathcal{D}.
\end{aligned}
\]
\end{proof}

\section{Global existence and decay}\label{se5}
In this section, we establish the \textit{a priori estimates} and obtain the well-posedness.
Our strategy to prove the global well-posedness is the usual continuity argument. We will first need to prove the uniform boundedness for any finite time interval $[0,T]$, which is concluded in our a priori estimates.
\subsection{A priori estimates}
Our a priori estimates is one ingredient required to develop the global well-posedness. It contains the decay of energy and the finite integral of dissipation.
\begin{theorem}\label{thm:priori}
  Suppose $(u,p,\eta)$ satisfies \eqref{eq:new_NS}. Then for any finite interval $[0,T]$, there exists a universal constant $\delta\in(0,1)$ such that, if
  \[
  \sup_{0\le t\le T}\mathcal{E}(t)\le\delta,
  \]
there exists a universal constant $\sigma>0$ so that
\[
  \sup_{0\le t\le T}e^{\sigma t}\mathcal{E}(t)+\int_0^T\mathcal{D}(t)\,\mathrm{d}t\lesssim\mathcal{E}(0).
  \]
Here, $\delta$ and $\sigma$ are independent of $T$.
\end{theorem}
\begin{proof}
  From the definition of horizontal energy $\mathcal{E}_\shortparallel$ and dissipation $\mathcal{D}_\shortparallel$,
and Theorem \ref{thm:enhanced_dissipation},
we see $\mathcal{E}\lesssim \mathcal{E}_\shortparallel+\mathcal{E}^2$ and $\mathcal{D}\lesssim \mathcal{D}_\shortparallel+\mathcal{E}\mathcal{D}$.
Consequently, after taking $\delta$ sufficiently small,
the terms $\mathcal{E}^2$ and $\mathcal{E}\mathcal{D}$ can be absorb into $\mathcal{E}$ and $\mathcal{D}$,
 respectively. Consequently, we have $\mathcal{E}\lesssim\mathcal{E}_\shortparallel$ and $\mathcal{D}\lesssim\mathcal{D}_\shortparallel$.
   We plunge this structure into the horizontal energy-dissipation inequality in Theorem \ref{thm:horizontal} and obtain
 \[
  \f{d}{dt}\left(\mathcal{E}_\shortparallel-\mathcal{P}\right)+\mathcal{D}_\shortparallel\lesssim\mathcal{E}^{1/2}\mathcal{D}_\shortparallel.
  \]
After restricting  $\delta$ small enough, we get
\beq\label{est:e_d1}
  \f{d}{dt}\left(\mathcal{E}_\shortparallel-\mathcal{P}\right)+\f12\mathcal{D}_\shortparallel\le0.
  \eeq
  So we need to ensure that the terms differentiated is positive. By the H\"older inequality, Sobolev inequalities and usual trace theory, we have
\[
\mathcal{P}=\int_\Om p\dive_{\pa_t\mathcal{A}}uJ\lesssim  \|p\|_1\|u\|_2 \left\|\pa_t\eta \right\|_{3/2}\lesssim\mathcal{E}^{1/2}\mathcal{E}_\shortparallel.
 \]
As a result, $\mathcal{P}\lesssim \delta^{1/2}\mathcal{E}_\shortparallel$, and hence
\[
\mathcal{E}_\shortparallel-\mathcal{P}\gtrsim \left(1-\delta^{1/2} \right)\mathcal{E}_\shortparallel,
\]
provided that we might restrict $\delta$ necessarily again. Also, it is trivial to see that
$\mathcal{E}_\shortparallel-\mathcal{P}\lesssim \mathcal{E}_\shortparallel$. Therefore, $\mathcal{E}_\shortparallel-\mathcal{P}$ is equivalent to $\mathcal{E}_\shortparallel$.

From the definition of $\mathcal{E}$ and $\mathcal{D}$ in \eqref{def:energy} and \eqref{def:dissipation}, it is obvious that
\[
\mathcal{D}\gtrsim\mathcal{E}_\shortparallel\gtrsim \mathcal{E}_\shortparallel-\mathcal{P}>0.
\]
Consequently, we can choose a universal constant $\sigma>0$ such that \eqref{est:e_d1} becomes
   \[
   \f{d}{dt}\left(\mathcal{E}_\shortparallel-\mathcal{P}\right)+\sigma\left(\mathcal{E}_\shortparallel-\mathcal{P}\right)\le0.
   \]
Then Gronwall's inequality implies
\[
\mathcal{E}_\shortparallel(t)\lesssim\mathcal{E}_\shortparallel(t)-\mathcal{P}(t)
\lesssim e^{-\sigma t} \left(\mathcal{E}_\shortparallel(0)-\mathcal{P}(0) \right)\lesssim e^{-\sigma t}\mathcal{E}(0).
 \]
 Therefore, we obtain a uniform bound
\[
\sup_{0\le t\le T}e^{\sigma t}\mathcal{E}(t)\lesssim\mathcal{E}(0).
\]
\eqref{est:e_d1} also implies
   \[
   \f{d}{dt}\left(\mathcal{E}_\shortparallel-\mathcal{P}\right)+C\mathcal{D}\le0,
   \]
   for some constant $C>0$. Then integrating over $[0,T]$ implies
   \[
   \int_0^T\mathcal{D}\lesssim\mathcal{E}_\shortparallel(0)\lesssim\mathcal{E}(0).
   \]
\end{proof}
\subsection{Local well-posedness}
Let $u_0\in H^2(\Om)$, $\eta_0\in H^3(\Sigma)$ and also satisfy the compatibility conditions,
we need to construct $\pa_tu(0)$, $p(0)$ and $\pa_t\eta(0)$ appearing in $\mathcal{E}(0)$.
We set $\pa_t\eta(0)=u_0\cdot\mathcal{N}_0$.
From the Sobolev embedding theory, we see that $\pa_t\eta(0)\in H^{1/2}(\Sigma)$.
Then we set $p(0)\in H^1(\Om)$ is the weak solution of
\beq
\left\{
\begin{aligned}\label{eq5.2v10}
  &-\dive_{\mathcal{A}_0} \left(\nabla_{\mathcal{A}_0} p(0) - \pa_t\bar{\eta}(0)K(0)W\pa_3u_0 \right) =\dive_{\mathcal{A}_0} (R(0)u_0)\in H^0(\Om),\\
  &p(0)=\left[ \left(\eta_0-\mathcal{H}(0) \right)\mathcal{N}_0 +\mathbb{D}_{\mathcal{A}_0}u_0\mathcal{N}_0\right]\cdot
\frac{ \mathcal{N}_0}{ \left|\mathcal{N}_0 \right|}\in H^{1/2}(\Sigma),\\
  &\left(\nabla_{\mathcal{A}_0}p(0)- \pa_t\bar{\eta}(0)K(0)W\pa_3u_0 \right)\cdot\nu
=\Delta_{\mathcal{A}_0}u_0\cdot\nu\in H^{1/2}(\Sigma_b),
\end{aligned}
\right.
\eeq
where $R=\pa_tMM^{-1}$ with the matrix $M=J\nabla\Phi$ that guarantees that $\pa_tu-Ru$ is $\dive_{\mathcal{A}}$ free. We also set $\pa_tu(0)=\Delta_{\mathcal{A}_0}u_0-\nabla_{\mathcal{A}_0}p(0)+\pa_t\bar{\eta}(0)K(0)W\pa_3u_0\in H^0(\Om)$. The elliptic equations \eqref{eq5.2v10} are well solved by the theory of Poisson equations.

Then the local well-posedness is stated as follows.
\begin{theorem}\label{thm:local}
  Assume that $\eta_0+1>\delta_0>0$ for some $\delta_0>0$. Suppose $(u_0,\eta_0)$ satisfies $\|u_0\|_{H^2(\Om)}^2+ \|\eta_0\|_{H^3(\Sigma)}^2<\varepsilon$ for some small $\varepsilon>0$ and the compatibility condition
  \eqref{cond:compatibility}.
  Then there exists a positive constant $\bar{T}<1$ such that for $0<T<\bar{T}$,
the system \eqref{eq:new_ns} admits a unique strong solution $(u,p,\eta)$ achieving the initial data and
  \beq
\sup_{0\le t\le T}\mathcal{E}(t)+\int_0^T\mathcal{D}(t)\,\mathrm{d}t+ \left\|\pa_t^2u \right\|_{(\mathcal{X}_T)^\ast}^2\le \varepsilon,
  \eeq
  where the space is defined to be
  \[
 \mathcal{X}_T=L^2([0,T];\mathcal{X}(t)), \text{ with } \mathcal{X}(t):= \left\{u\in \mathcal{H}^1(t): \dive_{\mathcal{A}(t)}u=0 \right\}.
  \]
Furthermore, the mapping $\Phi$ is a $C^1$ diffeomorphism.
\end{theorem}
The above theorem can be proved by a standard method based on the idea in \cite{ZhI17} coupled with the elliptic theory in Theorem \ref{thm:elliptic}. So we ignore to iron out the details here and only give a sketch.
\begin{proof}
We construct $ \left(u^0, \eta^0 \right)$ by \cite[Theorem A.5]{GT1} as a start point and then construct the approximate solutions $ \left(u^m, p^m, \eta^m \right)$ for $m\ge1$, by the following linear iterated system
  \beq\label{eq:iteration}
\left\{
\begin{aligned}
&\pa_tu^m+\nabla_{\mathcal{A}^{m-1}} p^m -\Delta_{\mathcal{A}^{m-1}}u^m =F^1 \left(u^{m-1},\eta^{m-1} \right)\ &\text{in}&\ \Om,\\
  &\dive_{\mathcal{A}^{m-1}}u^m=0\ &\text{in}&\ \Om,\\
  &S_{\mathcal{A}^{m-1}} \left(p^m,u^m \right)\mathcal{N}^{m-1} = \left(\eta^m-\Delta_\ast\eta^m \right)\mathcal{N}^{m-1}+F^3 \left(\eta^{m-1} \right)\ &\text{on}&\ \Sigma,\\
  &\pa_t\eta^m=u^m\cdot\mathcal{N}^{m-1} \ &\text{on}&\ \Sigma,\\
  &u^m=0\ &\text{on}&\ \Sigma_b,
\end{aligned}
\right.
\eeq
where
\begin{align*}
 & F^1 \left(u^{m-1},\eta^{m-1} \right)   =\pa_t\bar{\eta}^{m-1}WK^{m-1} \pa_3u^{m-1}-u^{m-1}\cdot\nabla_{\mathcal{A}^{m-1}}u^{m-1},\\
& F^3(\eta^{m-1})   =-\left(\nabla_\ast\cdot\left(\f{\nabla_\ast\eta^{m-1}}{\sqrt{1+ \left|\nabla_\ast\eta^{m-1} \right|^2}}
-\nabla_\ast\eta^{m-1}\right)\right)\mathcal{N}^{m-1},
\end{align*}
and $\mathcal{A}^{m-1}$, $\mathcal{N}^{m-1}$, $K^{m-1}$, $\mathcal{H}^{m-1}$ are determined in terms of $\eta^{m-1}$, with the initial data $\left(u^m(0), \eta^m(0) \right)= \left(u_0,\eta_0 \right)$.
Then we can apply the standard method to prove the limit of $\left(u^m, p^m, \eta^m \right)$ is the solution of \eqref{eq:new_ns}.

Hence the solvable of \eqref{eq:iteration} is reduced to solving the linear problem
\beq
\left\{
\begin{aligned}
  &\pa_tv+\nabla_{\mathcal{A}}q-\Delta_{\mathcal{A}}v=F^1\ &\text{in}&\ \Om,\\
  &\dive_{\mathcal{A}}v=0\ &\text{in}&\ \Om,\\
  &S_{\mathcal{A}}(q,v)\mathcal{N}= \left(\xi -\Delta_\ast\xi \right)\mathcal{N}+F^3\ &\text{on}&\ \Sigma,\\
  &\pa_t\xi=v\cdot\mathcal{N}\ &\text{on}&\ \Sigma,\\
&v=0\ &\text{on}&\ \Sigma_b,
\end{aligned}
\right.
\eeq
where $\mathcal{A}$, $\mathcal{N}$ and $J$ determined in terms of $\eta$ are given and $F^i$, $i=1, 3, 5$ are given functions. Then as the procedure in \cite[Theorem 4.8]{ZhI17}, we apply the Galerkin method with finite time-dependent basis to construct the sequence of approximate solutions $\left\{v^n \right\}\in\mathcal{X}_m$ and $\xi^n=\eta_0+\int_0^tv^n(s)\cdot\mathcal{N}(s)\,\mathrm{d}s$ satisfying
  \beq\label{eq:app_1}
\begin{aligned}
\left(\pa_tv^n,w \right)_{\mathcal{H}^0}+ \f12 \left(v^n, w \right)_{\mathcal{H}^1}+ \left(\xi^n,w\cdot\mathcal{N} \right)_{H^1(\Sigma)}
= \left(F^1,w \right)_{\mathcal{H}^0}
 - \left(F^3-\Pi_0 \left(F^3(0)+\mathbb{D}_{\mathcal{A}_0} \left(\mathcal{P}^m_0u_0 \right)\mathcal{N}_0 \right),w \right)_{0,\Sigma}
\end{aligned}
\eeq
with $w\in \mathcal{X}_m$, where $\mathcal{X}_m=\text{span} \left\{w^j \right\}_{j=1}^m$,
and $ \left\{w^j \right\}_{j=1}^\infty$ is a countable basis of $H^2(\Om)\cap\mathcal{X}(t)$.
$\mathcal{P}^m_0$ is a projection from $H^2(\Om)$ to $\mathcal{X}_m(0)$.
Then \eqref{eq:app_1} is reduced to an integral-differential equation 
\[
    \dot{d}^n(t)+A(t)d^n(t)+\int_0^t C(t,s)d^n(s)\,\mathrm{d}s=\mathfrak{F}(t),
    \]
where $A(t)$ and $C(t,s)$ are $n\times n$ matrix-valued $C^1$ functions, which can be solved by the theory of integral-differential equations(see for instance \cite{GM70}).

Then we use the standard argument to bound the $v^n$ and $\xi^n$ in the topology  of $\mathcal{E}$ and $\mathcal{D}$, similar to the argument in the proof of \cite[Theorem 4.8]{GZ}, to close the proof.
\end{proof}
\begin{remark}
Although we only need the local theory for small initial data in our global result,
we note that the local well-posedness can be established for the large data of $(u_0,\eta_0)$.
The techniques used for large data is just a combination of $\varepsilon$-modification for the mapping $\Phi$,
with the modified elliptic estimates in Theorem \ref{thm:elliptic} depending on $\eta_0$, that is similar as in \cite{Wu}.
\end{remark}
\subsection{Global existence and decay}
\begin{proof}[Proof of Theorem \ref{thm:main}.]
The result is a consequence of continuity argument. Let $T^\ast$ be the supremum of $T$ such that Theorem \ref{thm:priori} holds. Suppose 
\[
\sup_{0\le t\le T}e^{\sigma t}\mathcal{E}(t)  +\int_0^T\mathcal{D}(t)\,\mathrm{d}t\le2\varepsilon
\]
for $\varepsilon$ sufficiently small and any $T<T^\ast$. Then the a priori estimates in Theorem \ref{thm:priori} implies
  \[
  \sup_{0\le t\le T}e^{\sigma t}\mathcal{E}(t)+\int_0^T\mathcal{D}(t)\,\mathrm{d}t\le C_1\mathcal{E}(0)
  \]
for a universal constant $C_1$.
The Local well-posedness of Theorem \ref{thm:local} implies $\mathcal{E}(0)\le C_2 \left(\|u_0\|_2^2+ \|\eta_0\|_3^2 \right)$ for a universal constant
$C_2$.
Then we can restrict the initial data $(u_0,\eta_0)$,
such that $\delta$ small enough, to ensure that
$C_1C_2 \left(\|u_0\|_2^2+\|\eta_0\|_3^2 \right)<\varepsilon$,
contradicting to the definition of $T^\ast$. Hence the continuity argument yields $T^\ast=\infty$.
\end{proof}

\textbf{Acknowledgment.}
Xing Cheng has been supported by the ``the Fundamental Research Funds for the Central Universities" (No.B210202147), and Yunrui Zheng is partially supported by NSF of China under Grant 11901350, and The Fundamental Research Funds of Shandong University under Grant 11140079614046.

\end{document}